\documentclass[12pt]{article}
\usepackage{amssymb}
\usepackage{amsmath}
\usepackage{amsfonts}
\usepackage{graphicx}
\newtheorem{theorem}{Theorem}

\newtheorem{definition}[theorem]{Definition}
\newtheorem{example}[theorem]{Example}

\newtheorem{lemma}[theorem]{Lemma}

\newtheorem{proposition}[theorem]{Proposition}
\newtheorem{remark}[theorem]{Remark}

\newenvironment{proof}[1][Proof]{\noindent\textbf{#1.} }{\ \rule{0.5em}{0.5em}}
\providecommand{\U}[1]{\protect\rule{.1in}{.1in}}
\newdimen\dummy
\dummy=\oddsidemargin
\begin{document}

\title{Some overview on unbiased interpolation \\and extrapolation designs}
\author{Michel Broniatowski$^{(1)}$ and Giorgio Celant$^{(2,\ast)}$\\$^{(1)}$LSTA, Universit\'{e} Pierre et Marie Curie, Paris, France\\$^{(2,\ast)}$Corresponding author, Dipartimento di Scienze Statistiche, \\Universit\`{a} degli Studi di Padova, Italy}
\maketitle

\begin{abstract}
This paper considers the construction of optimal designs due to Hoel and
Levine and Guest.\ It focuses on the relation between the theory of the
uniform approximation of functions and the optimality of the designs. Some
application to accelerated tests is also presented. The multivariate case is
also handled in some special situations.

AMS Classification: 62K05, 62N04

Key words: Optimal design, Extrapolation, Interpolation, Lagrange polynomials,
Legendre polynomials, Chebyshev points
\end{abstract}

\section{Introduction}

\subsection{Definition of the model and of the estimators}

We assume that we are given an interval where the explanatory variable $x$
takes its value; for simplicity assume that $x$ belongs to $\left[
-1,1\right]  .$ On this interval the response \ $Y$ can be observed. An
additive noise causes $Y$ to be only partly related to the input $x.$ This
noise is assumed to be independent upon the value of $x,$ which is commonly
referred to as a homoscedastic hypothesis on the model. For a given input $x$
the measure $Y(x)$ can be written as
\[
Y(x)=f(x)+\varepsilon
\]
where $f$ is some unknown function and the generic real valued random variable
$\varepsilon$ has some unknown distribution; however it is assumed that the
two first moment of $\varepsilon$ are finite.\ The function $f$ might be
defined on a larger interval than $\left[  -1,1\right]  .$ All possible
measurements of $f$ can only be achieved on $\left[  -1,1\right]  .$ It may
occur that we are interested in some estimation of $f(x)$ for some $x$ where
$f$ is not measured; when $x$ belongs to $\left[  -1,1\right]  $ this is an
interpolation problem. At times we may be interested in some approximation of
$f(x)$ for $x$ outside $\left[  -1,1\right]  $; this is an extrapolation problem.

We will discuss optimal designs in those contexts. Defining a design results
in a two fold description. Firstly it is based on a set of measurements
points, say $x_{0},..,x_{g-1}$ in $\left[  -1,1\right]  .$ Those are the nodes
of the design. Next, for any node $x_{j}$, we define an integer $n_{j}$ which
is the number of replicates of the measurement performed under the condition
$x_{j}.$ We thus inherit of the $n_{j}$ measurements $Y_{1}(x_{j}%
),..,Y_{n_{j}}(x_{j}).$ Those measurements are supposed to be independent.
Note that we do not assume any probabilistic structure on the $x_{j}$'s which
therefore will not be considered as sampled under any device. The $x_{j}$'s
are determined by the experimenter and their choice will follow from a
strictly deterministic procedure.

Obviously this simple model will allow for a simple estimate of $f(x_{j})$ for
all $j,$ assuming without loss of generality that the error $\varepsilon$ has
expectation $0.$

The design is therefore defined by the family of the nodes (their number $g$
is fixed by the experimenter), and the so called frequencies $n_{j}$'s, $0\leq
j\leq g-1.$.

Obviously the total number of experiments is limited, for reasons which have
to do with the context of the study. Those reasons might be related to the
cost of each individual experiment, or by other considerations. For example in
phase 1 clinical trials it is usually assumed that only very few patients can
be eligible for the trial. Call $n$ this number of trials to be performed. The
resulting constraint on the $n_{j}$'s is therefore%
\[
n_{0}+..+n_{g-1}=n.
\]
Let us now define formally this model.

For any $i=0,..,g-1$, $x_{i}$ is a node and $y_{k}(x_{i})$ is the $k-$th measurement of $Y(x_{i})$ when $k$ runs in $1,..,n_{i}.$ Therefore

\[
\left\{
\begin{array}
[c]{c}%
y_{1}\left(  x_{i}\right)  =f\left(  x_{i}\right)  +\varepsilon_{1,i}\\
......................\\
y_{n_{i}}\left(  x_{i}\right)  =f\left(  x_{i}\right)  +\varepsilon_{n_{i},i}%
\end{array}
\right.
\]
where $n_{i}>1$ and $n_{i}\in\mathbb{N},$ together with $n:=\sum_{i=0}^{g-1}n_{i}$ where $n$ is fixed$.$ Obviously
the r.v's $\varepsilon_{j,i}$, $1\leq j\leq n_{i}$ , $i=0,...,g-1,$ are not
observed. They are i.i.d. copies of a generic r.v. $\varepsilon.$ Furthermore $E\left(  \varepsilon\right)  =0$, $var\left(  \varepsilon\right)  =\sigma^{2}.$

We assume that $f$ is a polynomial with known degree $g-1$ $.$ Therefore it is
completely determined if known the values of $f$ in $g$ distinct points. Note
that the knowledge of $g$ is an important and strong assumption. Denote
further
\[
I:=\left\{  x_{0}<...<x_{g-1}\right\}  \subset\left[  -1,1\right]
\]
the family of nodes.

The aim of this chapter is to discuss the operational choice of the design; we
will thus propose some choices for the nodes and the so-called frequencies
$n_{j}/n$ which, all together , define the design. This will be achieved
discussing some notion of optimality.\\

When no random effect is present, existence and uniqueness of the solution of
the linear system with $g$ equations and $g$ variables $\theta:=\left(\theta_{0},...,\theta_{g-1}\right),$
\[
y\left(  x_{i}\right)  =\sum_{j=0}^{g-1}\theta_{j}x_{i}^{j},\left(
x_{0},...,x_{g-1}\right)  \in\left[  -1,1\right]  ^{g},0\leq i\leq g-1
\]
allow to identify the function $f$ at any $x$ in $%
\mathbb{R}
.$ Changing the canonical basis in $P_{g-1}\left(  X\right)  $ into the family
of \ the elementary Lagrange polynomials%
\begin{equation}
l_{i}\left(  x\right)  :=%
{\textstyle\prod\limits_{j=0,j\neq i}^{g-1}}
\frac{x-x_{j}}{x_{i}-x_{j}}\label{PolElemLagrange}%
\end{equation}
yields%
\[
f(x)=\sum_{i=0}^{g-1}f(x_{i})l_{i}(x).
\]

In the present random setting, $f(x_{i})$ is unknown. This suggests
to consider the estimator $\widehat{f}$ of $f$ defined by%

\[
\mathcal{L}_{n}(\widehat{f})(x):=\sum_{i=0}^{g-1}\widehat{f\left(
x_{i}\right)  }\text{ }l_{i}(x),
\]

where $\widehat{f\left(  x_{i}\right)  }$ \ denotes some estimate of $f$ on a
generic node $x_{i}.$

Turn to the estimate of $f(x_{i})$ namely the simplest one, defined by
\[
\widehat{f\left(  x_{i}\right)  }:=\overline{Y}\left(  x_{i}\right)
:=\frac{1}{n_{i}}\sum_{j=1}^{n_{i}}Y_{j}\left(  x_{i}\right)  ,
\]
which solves%

\[
\widehat{f\left(  x_{i}\right)  }=\arg\min_{\mu\in\mathbb{R}}\sum_{j=1}%
^{n_{i}}\left(  Y_{j}\left(  x_{i}\right)  -\mu\right)  ^{2},
\]
provides the optimal linear unbiased estimator of $f\left(  x_{i}\right)  .$
It follows that $\ \widehat{f\left(  x\right)  }:=\mathcal{L}_{n}%
(\widehat{f})(x)$ is unbiased since for all $x\in%
\mathbb{R}
,$%
\[
E\left(  \mathcal{L}_{n}(\widehat{f})(x)\right)  :=\sum_{i=0}^{g-1}E\left(
\widehat{f\left(  x_{i}\right)  }\right)  l_{i}(x)=\mathcal{L}_{n}%
(f)(x)=f\left(  x\right)  .
\]

Since $\mathcal{L}_{n}(\widehat{f})(x)$ is linear with respect to the
parameters $f\left(  x_{i}\right)  ,$ $i=0,...,g-1,$ using Gauss Markov Theorem, $\mathcal{L}_{n}(\widehat{f})(x)$ is optimal, i.e. has minimal variance. The variance of the estimator $\mathcal{L}_{n}(\widehat{f})(x)$ is
\begin{equation}
var\left(  \mathcal{L}_{n}(\widehat{f})(x)\right)  =var\left(  \sum
_{i=0}^{g-1}var\left(  \widehat{f\left(  x_{i}\right)  }\right)  \left(
l_{i}(x)\right)  ^{2}\right)  =\sigma^{2}\sum_{i=0}^{g-1}\frac{\left(
l_{i}(x)\right)  ^{2}}{n_{i}},\label{varx}%
\end{equation}
which depends explicitly on the frequency $n_{i}$ of the observations of $f$ on the nodes $x_{i}$'s.

We now proceed to the formal definition of a design. The set
\[
\left\{  \left(  \left(  n_{0},...,n_{g-1}\right)  ,\left(  x_{0}%
,...,x_{g-1}\right)  \right)  \in%
\mathbb{N}
^{g}\times\left[  -1,1\right]  ^{g-1}:n:=\sum_{i=1}^{g-1}n_{i},\text{ }n\text{
fixed}\right\}
\]
determines a discrete probability measure $\xi$ with support $I,$ a finite
subset in $\left[  -1,1\right]  ,$ by%
\[
\xi\left(  x_{i}\right)  :=\frac{n_{i}}{n},\text{ }i=0,...,g-1.
\]

Turning to (\ref{varx}) we observe that the accuracy of the design \ depends
on the point $x$ where the variance of $\mathcal{L}_{n}(\widehat{f})(x)$ is calculated.\\

Since all estimators of the form $\mathcal{L}_{n}(\widehat{f})(x)$ are
unbiased, their accuracy depend only on their variance, which in turn depends
both on $x$ and on the measure $\xi.$ Optimizing on $\xi$ for a given $x$
turns to an optimal choice for $I$ and for the family of the $n_{i} $'s, under
the constraint
\[
n_{0}+..+n_{g-1}=n.
\]

Such designs $\xi_{x}$ are called \ Hoel-Levine extrapolation designs when $x$
lies outside $\left[  -1,1\right]  $. When $x$ belongs to $\left[
-1,1\right]  $ then clearly the optimal design for the criterion of the
variance of the estimator of $f(x)$ results in performing all the $n$
measurements at point $x.$ The non trivial case is when the optimality is
defined through a control of the uniform variance of the estimator of $f(x)$,
namely when $\xi$ should minimize%
\[
\sup_{x\in\left[  -1,1\right]  }var_{\xi}\mathcal{L}_{n}(\widehat{f})(x).
\]
Those designs $\xi$ are called interpolation designs, or Guest or Legendre designs.\\

The notation to be kept is as follows. The set of all probability measures on
the interval $\left[  -1,1\right]  $ supported by $g$ distinct points in the
interval $\left[  -1,1\right]  $ is denoted $\mathcal{M}_{\left[  -1,1\right]
}^{\ast}$, which therefore is the class of all designs.

The purpose of this paper is to present a unified view on this classical field
which links the theory of the uniform approximation of functions and the
statistical theory of experimental designs, unifying notation and concepts.
Its content comes mostly from the seminal works by Hoel and Levine
\cite{Hoel1964}, Guest \cite{Guest1958}. \ Some other important references in
the context is Kiefer and Wolfowicz \cite{Kiefer1964} and Studden
\cite{Studden1968}.

\subsection{Some facts from the theory of the approximation of
functions}

We briefly quote some basic facts from classical analysis, to be used
extensively. For fixed $n$ let \ $\mathcal{P}_{n}$ denote the class of all
polynomials with degree less or equal $n$ defined on $\left[  -1,1\right]  $
and let $P_{n}(x):=a_{0}+...+a_{n}x^{n}$. Some coordinates of the vector of
coefficients $\left(  a_{0},...,a_{n}\right)  $ can take value $0.$

For a continuous function $f$ defined on $\left[  -1,1\right]  $ denote
\[
e:=f-P_{n}%
\]
which is a continuous function on $\left[  -1,1\right]  ,$ as is $\left\vert
e\right\vert .$

Applying Weierstrass Theorem it follows that $\left\vert e\right\vert $
attains its maximal value in $\left[  -1,1\right]  $, for at least one point $x.$ We denote
\[
\mu=\mu\left(  a_{0},a_{1},...,a_{n}\right)  :=\max_{x\in\left[  -1,1\right]
}\left\vert e\left(  x\right)  \right\vert \geq0
\]
which yields to define
\begin{equation}
\alpha:=\inf_{\left(  a_{0},a_{1},...,a_{n}\right)  }\mu\left(  a_{0}%
,a_{1},...,a_{n}\right)  ,\label{alfa}
\end{equation}
the minimum of the uniform error committed substituting $f$ by a polynomial in
$\mathcal{P}_{n}$ . This lower bound exists.

\begin{definition}
A polynomial $P^{\ast}$ in $\mathcal{P}_{n}$ with coefficients $\left(
a_{0}^{\ast},...,a_{n}^{\ast}\right)  $ such that
\[
\sup_{x\in\left[  -1,1\right]  }\left\vert f\left(  x\right)  -P^{\ast}\left(
x\right)  \right\vert =\alpha\text{,}%
\]
where $\alpha$ is defined in (\ref{alfa}) is called a best approximating
polynomial of $f$ with degree $n$ in the uniform sense$.$
\end{definition}

This best approximation thus satisfies
\[
P^{\ast}\left(  x\right)  :=\arg\inf_{P_{n}\in\mathcal{P}_{n}}\sup
_{x\in\left[  -1,1\right]  }\left\vert e\left(  x\right)  \right\vert .
\]

The polynomial $P^{\ast}$ may be of degree less than $n.$ Define
\begin{equation}
e^{\ast}\left(  x\right)  :=f\left(  x\right)  -P^{\ast}\left(  x\right)
\label{Phi_n^*}
\end{equation}
and
\begin{equation}
E\left(  f\right)  :=\max_{x\in\left[  -1,1\right]  }\left\vert e^{\ast
}\left(  x\right)  \right\vert .\label{E_n(f)}%
\end{equation}

\subsubsection{Existence of the best approximation}

The following result answers the question of attainment for the least uniform
error when approximating $f$ by a polynomial in $\mathcal{P}_{n}$.

\begin{theorem}
\label{ThmExistenceApprox}Let $f$ be some continuous function defined on
$\left[  -1,1\right]  .$ For any integer $n$ there exists a unique $P^{\ast}$
in $\mathcal{P}_{n}$ such that $\left\Vert f-P^{\ast}\right\Vert _{\infty
}=\inf_{P_{n}\in\mathcal{P}_{n}}$ $\left\Vert f-P_{n}\right\Vert _{\infty}. $
\end{theorem}

\subsubsection{Uniqueness of the best approximation}

The proof of the uniqueness Theorem strongly relies on a Theorem by Chebyshev
which explores the number of changes of the sign of the error, which we state
now. Consider $P^{\ast}\left(  x\right)  $ and $E\left(  f\right)  $ as
defined in (\ref{Phi_n^*}) and (\ref{E_n(f)}).

\begin{theorem}
\label{ThmBorel-Tchebicheff}(Borel-Chebyshev) \ \ \ \ \ A polynomial $P^{\ast
}$ in $\mathcal{P}_{n}$ is a best uniform approximation of a function $f$ in
$\mathcal{C}^{\left(  0\right)  }\left[  -1,1\right]  $ \ if and only if the
function $x\rightarrow e^{\ast}\left(  x\right)  $ equals $E(f)$ with
alternating signs for at least $n+2$ values of $x$ in $\left[  -1,1\right]  .$
\end{theorem}

\begin{theorem}
\label{TmmUnicite}Let $f$ be a continuous function defined on $\left[
-1,1\right]  $.The best uniform approximating polynomial is unique.
\end{theorem}

For a complete treatment of the above arguments, see e.g. \cite{Dzyadyk2008},
\cite{Rynne2008} and \cite{Kolmogorov1981}.

\section{Optimal extrapolation designs; Hoel Levine or Chebyshev designs}

We consider the problem of approximating $f(x)$ for any fixed $x$ in the
interval $\left[  c,-1\right]  $ for some $c<-1.$ More generally the optimal
design which is obtained in this section is valid for any $c$ such that
$\left\vert c\right\vert >1.$

As seen previously, since $\mathcal{L}_{n}(\widehat{f})(x)$ is an unbiased
estimate of $f(x)$ a natural criterion for optimality in the class of all
unbiased linear estimates is the variance.\\

We therefore consider the problem
\[
\xi_{x}^{\ast}:=\arg\min_{\xi\in\mathcal{M}_{\left[  -1,1\right]  }^{\ast}%
}var\left(  \mathcal{L}_{n}(\widehat{f})(x)\right)
\]
to which we will propose a suboptimal solution.

Denoting generically $\mathbf{n}^{\ast}:=\left(  n_{0}^{\ast},...,n_{g-1}%
^{\ast}\right)  $ and $(x^{\ast}):=\left(  x_{0}^{\ast},...,x_{g-1}^{\ast
}\right)  $ , $\mathbf{n}:=\left(  n_{0},...,n_{g-1}\right)  \in$ $%
\mathbb{N}
^{g}$ and $(x):=\left(  x_{0},...,x_{g-1}\right)  $ this problem can thus be written as
\[
\left\{
\begin{array}
[c]{c}%
\left(  \mathbf{n}^{\ast},(x^{\ast})\right)  =\arg\min_{\left(  \mathbf{n}%
,\left(  x\right)  \right)  }\sum_{i=0}^{g-1}\frac{\left(  l_{i}(x)\right)
^{2}}{n_{i}}\\
\mathbf{n}:\text{ }\sum_{i=0}^{g-1}n_{i}=n\text{, }n\text{ fixed }\\
-1\leq x_{0}<...<x_{g-1}\leq1,\left(  x\right)  \in\left[  -1,1\right]  ^{g}%
\end{array}
\right.  .
\]

This is an integer programming problem (w.r.t. $\mathbf{n}$) with an
inequality constraint in $
\mathbb{R}
$. As commonly done we find a proxy to the solution, considering the $n_{i}$'s as real numbers. The resulting solution $\left(  n_{0}^{\ast},...,n_{g-1}^{\ast}\right)  $ will be substituted by the integer part of each of the $n_{i}^{\ast}$'s$.$

We therefore get to the following optimization problem in the $2g$ real variables
\begin{equation}
\left\{
\begin{array}
[c]{c}%
\left(  \left(  w_{0}^{\ast},...,w_{g-1}^{\ast}\right)  ,\left(  x_{0}^{\ast
},...,x_{g-1}^{\ast}\right)  \right)  =\arg\min_{\left(  \left(
w_{0},...,w_{g-1}\right)  ,\left(  x_{0},...,x_{g-1}\right)  \right)  }%
\sum_{i=0}^{g-1}\frac{\left(  l_{i}(x)\right)  ^{2}}{w_{i}}\\
\left(  w_{0},...,w_{g-1}\right)  \in%
\mathbb{R}
^{g},w>0\text{, }\sum_{i=0}^{g-1}w_{i}=n\text{, }n\text{ fixed }\\
-1\leq x_{0}<...<x_{g-1}\leq1,\left(  x_{0},...,x_{g-1}\right)  \in\left[
-1,1\right]  ^{g}.
\end{array}
\right.  .\label{Opt1}
\end{equation}

Since the mapping
\[
\left(  \left(  w_{0},...,w_{g-1}\right)  ,\left(  x_{0},...,x_{g-1}\right)
\right)  \rightarrow\sum_{i=0}^{g-1}\frac{\left(  l_{i}(x)\right)  ^{2}}
{w_{i}}
\]
is continuous, the optimization problem (\ref{Opt1}) can be solved in a two steps procedure.

The principle is as follows: fix the vector \underline{$\mathbf{x}$}$:=\left(
x_{0},...,x_{g-1}\right)  ^{\text{ }\prime}$ and look for the minimum with respect to the vector \underline{$\mathbf{w}$}$:=\left(  w_{0},...,w_{g-1}%
\right)  ^{\text{ }\prime}$. Once obtained the optimizer
\underline{$\mathbf{w}$}$^{\ast}$ vary \underline{$\mathbf{x}$} and determine the resulting minimum value of the variance for fixed \underline{$\mathbf{w}$%
}$=$\underline{$\mathbf{w}$}$^{\ast}$.

Define therefore , with fixed \underline{$\mathbf{x}$}
\begin{equation}
\left(  2_{a}\right)  \left\{
\begin{array}
[c]{c}%
\min_{\left(  w_{0},...,w_{g-1}\right)  \in%
\mathbb{R}
^{g}}\sum_{i=0}^{g-1}\frac{\left(  l_{j}\left(  x\right)  \right)  ^{2}}%
{w_{j}}\\
\sum_{i=0}^{g-1}w_{i}=n\text{, }n\text{ fixed}%
\end{array}
\right.  .\label{Pb1}%
\end{equation}

Denote $\left(  w_{0}^{\ast},...,w_{g-1}^{\ast}\right)  $ the
solution of this problem.\\

The optimal design $\xi^{\ast}$ results then as the solution of the following
problem, assuming in the present case that $x<x_{0}$
\begin{equation}
\left(  2_{b}\right)  \left\{
\begin{array}
[c]{c}%
\min_{\left(  x_{0},...,x_{g-1}\right)  \in%
\mathbb{R}
^{g}}\sum_{i=0}^{h}\frac{\left(  l_{j}\left(  x\right)  \right)  ^{2}}%
{w_{j}^{\ast}}\\
x<-1<x_{0}<...<x_{g-1},\left\{  x_{0},...,x_{g-1}\right\}  \subset\left[
-1,1\right]
\end{array}
\right.  .\label{Pb2}%
\end{equation}

\textbf{Step 1.} We handle\textbf{\ }Problem (\ref{Pb1}).

\begin{proposition}
The solution of problem (\ref{Pb1}) exists and is unique. It is given
by
\[
w_{j}^{\ast}:=n\frac{\left\vert l_{j}\left(  x\right)  \right\vert }%
{\sum_{j=0}^{g-1}\left\vert l_{j}\left(  x\right)  \right\vert }\text{,
}j=0,...,g-1.
\]
\end{proposition}

\begin{proof}
Applying the Karush-Kuhn-Tucker Theorem (see e.g. \cite{Bazaraa2006}), we obtain
\[
\left\{
\begin{array}
[c]{c}%
\frac{\partial}{\partial w_{j}}\left(  \sum_{i=0}^{g-1}\frac{\left(
l_{i}\left(  x\right)  \right)  ^{2}}{w_{i}}+\lambda\left(  \sum_{i=0}%
^{g-1}w_{i}-n\right)  \right)  =0\\
\lambda\left(  \sum_{i=0}^{g-1}w_{i}-n\right)  =0\\
\lambda\geq0
\end{array}
\right.  ,
\]
\[
\frac{\partial}{\partial w_{j}}\left(  \sum_{i=0,i\neq j}^{g-1}\frac{\left(
l_{i}\left(  x\right)  \right)  ^{2}}{w_{j}}+\frac{\left(  l_{j}\left(
x\right)  \right)  ^{2}}{w_{j}}+\lambda\left(  \sum_{i=0.i\neq j}^{g-1}
w_{i}+w_{j}-n\right)  \right)  =0,
\]
\[
\lambda=\left(  \frac{l_{j}\left(  x\right)  }{w_{j}}\right)  ^{2},
\]
\[
w_{j}=\frac{\left\vert l_{j}\left(  x\right)  \right\vert }{\sqrt{\lambda}}.
\]

Since $\sum_{i=0}^{g-1}w_{i}=n$, we get
\[
\left(  \frac{l_{j}\left(  x\right)  }{w_{j}}\right)  ^{2}=\lambda=\left(
\frac{\sum_{j=0}^{g-1}\left\vert l_{j}\left(  x\right)  \right\vert }
{n}\right)  ^{2}.
\]

Finally we get the solution of the problem $\left(2_{a}\right)$, namely
\[
w_{j}^{\ast}:=n\frac{\left\vert l_{j}\left(  x\right)  \right\vert }%
{\sum_{j=0}^{g-1}\left\vert l_{j}\left(  x\right)  \right\vert }\text{,
}j=0,...,g-1.
\]
\end{proof}

\bigskip

\textbf{Step 2. }Solving Problem (\ref{Pb2}) is more tedious and requires some
technical arguments, which we develop now.

Substituting $n_{j}$ in (\ref{varx}) \ by $w_{j}^{\ast}$ the variance of
$\mathcal{L}_{n}(\widehat{f})(x)$ becomes%
\[
Var\mathcal{L}_{n}(\widehat{f})(x)=\sigma^{2}\left(  \sum_{i=0}^{g-1}%
\left\vert l_{i}\left(  x\right)  \right\vert \right)  ^{2}.
\]
Hence minimizing this variance under the nodes turns into the following problem
\begin{equation}
\left\{
\begin{array}
[c]{c}%
\min_{\left(  x_{0},...,x_{g-1}\right)  \in%
\mathbb{R}
^{g}}\sum_{i=0}^{g-1}\left\vert l_{i}\left(  x\right)  \right\vert \\
x<x_{0}<...<x_{g-1},\left\{  x_{0},...,x_{g-1}\right\}  \subset\left[
-1,1\right]
\end{array}
\right. \label{PbOpt}%
\end{equation}
$.$

Since $x<x_{0}<...<x_{g-1},$ we have
\begin{equation}
\left\vert l_{i}\left(  x\right)  \right\vert =\left(  -1\right)  ^{i}%
l_{i}\left(  x\right)  .\label{(3)}%
\end{equation}

We consider a suboptimal solution to Problem (\ref{PbOpt}). Indeed no general
solution presently exists to the determination of a system of nodes which
minimizes the evaluation of the Lebesgue function $t\rightarrow\sum
_{i=0}^{g-1}\left\vert l_{i}\left(  t\right)  \right\vert $ at some fixed
point $t=x<-1.$ $\ $

Define
\begin{equation}
t\rightarrow T_{g-1}\left(  t\right)  :=\sum_{i=0}^{g-1}\left(  -1\right)
^{i}l_{i}\left(  t\right) \label{PolT}%
\end{equation}
a polynomial with degree $g-1$ defined on $\mathbb{R}$ $.$ This polynomial
does not depend any longer of $x$ but only on the nodes$.$ Up to the
multiplicative constant $\sigma$ it coincides with the standard deviation of
$\mathcal{L}_{n}(\widehat{f\text{ }})(x)$ when evaluated at point $x.$ We
provide the optimal choice of the nodes under this representation.

Observe the properties of $T_{g-1}$ on $\left[  -1,1\right]  .$ It holds%
\[
T_{g-1}(x_{j})=(-1)^{j}%
\]
for all $j$ between $0$ and $g-1.$ Hence $T_{g-1}$ takes its extreme values in
$g$ points; among those points, all interior ones in $\left[  -1,1\right]  $
are points where $T_{g-1}$ changes curvature. Those are obtained as the roots
of the first derivative of $T_{g-1};$ indeed $t\rightarrow Tq_{g-1}\left(
t\right)  $ is monotonous when $t$ does not belong to $\left[  -1,1\right]  .$
Hence we obtain that $g-2$ points among the $x_{i}$'s are interior points in
$\left[  -1,1\right]  .$ It follows that the two remaining ones are $-1$ and
$1.$

We now identify $T_{g-1}$ and therefore the nodes.

\bigskip

\begin{lemma}
The polynomial $T_{g-1}$ is the solution of the differential equation
\end{lemma}

\bigskip%
\begin{equation}
1-T_{g-1}^{\text{ }2}(x)=\frac{1}{\left(  g-1\right)  ^{2}}(1-x^{2})\left(
\frac{dT_{g-1}\left(  x\right)  }{dx}\right)  ^{2}.\label{equadiff}%
\end{equation}

\begin{proof}
It should hold%
\[
\left\{
\begin{array}
[c]{c}%
T_{g-1}\left(  x_{j}\right)  =\left(  -1\right)  ^{j}\\
\left(  \frac{dT_{g-1}\left(  x\right)  }{dt}\right)  _{x=x_{j}}=0,\text{
}j=1,...,g-2
\end{array}
\right.  .
\]
Observe that

a) the degree of $(1-T_{g-1}^{\text{ }2}(x))$ equals $2\left(  g-1\right)  $
and $1-T_{g-1}^{\text{ }2}(x_{j})=0,$ which implies that the set of roots of
$1-T_{g-1}^{\text{ }2}(x)$ are $\{x_{0},...,x_{g-1}\}$;

b) since $T_{g-1}(x)=a_{0}+a_{1}x+...+a_{g-1}x^{g-1},$ it follows that
$\frac{dT_{g-1}(x)}{dx}=a_{1}+a_{2}x+...+\left(  g-1\right)  a_{g-1}t^{g-2},$
and\
\[
\left(  \frac{dT_{g-1}(x)}{dx}\right)  ^{2}=\beta_{1}+...+\beta_{g-2}%
x^{2\left(  g-2\right)  -2}.
\]
This implies that%

\[
degree(1-x^{2})\left(  \frac{dT_{g-1}(x)}{dx}\right)  ^{2}\leq2\left(
g-1\right)  .\
\]

Furthermore
\[
\ (1-x^{2})\left(  \frac{dT_{g-1}(x)}{dx}\right)  ^{2}=0
\]
when $x\in\{-1,1\}=\{x_{0},x_{g-1}\},$ and
\[
T_{g-1}^{\prime}(x)=0\text{ when }x\in\{x_{1},...,x_{g-2}\}.
\]
From a) and b) we deduce that the polynomials $1-T_{g-1}^{\text{ }2}(x)$ and
$(1-x^{2})\left(  \frac{dT_{g-1}(x)}{dx}\right)  ^{2}$ share the same roots
and the same degree, hence are equal up to a constant $K,$ i.e.%
\begin{equation}
1-T_{g-1}^{\text{ }2}(x)=K(1-t^{2})\left(  \frac{dT_{g-1}(x)}{dx}\right)
^{2}\text{.}\label{(1)}%
\end{equation}
We determine $K$. Let
\[
T_{g-1}(x)=\sum_{i=0}^{g-1}a_{i}x^{i}.
\]
Then the coefficient of $x^{2\left(  g-1\right)  }$ equals $a_{g-1}^{2}$. Also
since
\[
\frac{dT_{g-1}(x)}{dx}=\sum_{i=1}^{g-2}ia_{i}x^{i-1}%
\]
then the term with highest degree in $\left(  \frac{dT_{g-1}(x)}{dx}\right)
^{2}$ is
\[
\left(  g-1\right)  ^{2}a_{g-1}^{2}x^{2\left(  g-1\right)  -2}.
\]
By (\ref{(1)}) it should hold that the coefficient of greatest degree of
$1-T_{g-1}^{\text{ }2}(x)$ equals the corresponding one of $K(1-x^{2})\left(
\frac{dT_{g-1}(x)}{dx}\right)  ^{2}$. \bigskip This means%
\[
-a_{g-1}^{2}x^{2\left(  g-1\right)  }=-Kx^{2}\left(  g-1\right)  ^{2}%
a_{g-1}^{2}x^{2\left(  g-1\right)  -2}%
\]
which yields
\[
1-T_{g-1}^{\text{ }2}(x)=\frac{1}{\left(  g-1\right)  ^{2}}(1-x^{2})\left(
\frac{dT_{g-1}(x)}{dx}\right)  ^{2}\text{.}%
\]
This is a differential equation with separable variables with solution the
$T_{g-1}(x)$ which we look for.
\end{proof}

We first solve the differential equation (\ref{equadiff}).

\begin{lemma}
\bigskip The solution of (\ref{equadiff}) under the boundary conditions
$T_{g-1}(-1)$ $=$ $T_{g-1}(-1)=1$ is%
\[
T_{g-1}(x)=\cos\left(  \left(  g-1\right)  \arccos x\right)
\]
\end{lemma}

\begin{proof}
Denote $y=T_{g-1}(\mathsf{v}).$ For $\mathsf{v}=x_{g-1}=1$, it should hold
$y=T_{g-1}(1)=1$. Therefore at $\mathsf{v}=1$, $T_{g-1}$ has a maximum.\ This
proves that for $\mathsf{v}<1$, with $\mathsf{v}$ close to $1$, $y=T_{g-1}%
(\mathsf{v})$ \ is increasing . Let $\mathsf{v}^{\ast}$ the point to the left
of $x_{g-1}$ such that from $\mathsf{v}^{\ast}$ to $x_{g-1}$ the function $y$
is always increasing, and is not increasing before $\mathsf{v}^{\ast}$.
Clearly $\mathsf{v}^{\ast}$ is a minimizer of $T_{g-1\text{ }}$ and
$y(\mathsf{v}^{\ast})=-1$. Therefore on $[\mathsf{v}^{\ast},1]$ the first
derivative $y^{\prime}$ is positive.

We can write, therefore%
\[
\sqrt{(y^{\prime})^{2}}=|y^{\prime}|=y^{\prime}\text{.}%
\]

It follows that the equation $1-y^{2}=\frac{1}{\left(  g-1\right)  ^{2}%
}(1-\mathsf{v}^{2})(y^{\prime})^{2}$ may be written as
\[
\frac{g-1}{\sqrt{1-\mathsf{v}^{2}}}=\frac{y^{\prime}}{\sqrt{1-y^{2}}}\text{.}
\]

Take the primitive on all terms, namely
\[
\left(  g-1\right)  \int\frac{d\mathsf{v}}{\sqrt{1-\mathsf{v}^{2}}}=\int%
\frac{y^{\prime}}{\sqrt{1-y^{2}}}d\mathsf{v}+c\text{. }%
\]

Some attention yields%
\[
\left(  g-1\right)  \arccos\mathsf{v}=\arccos y(\mathsf{v})+c^{\prime}\text{.}%
\]

Hence
\[
\cos(\left(  g-1\right)  \arccos\mathsf{v})=\cos(\arccos y+c^{\prime})\text{.}%
\]

Take $\mathsf{v}=1$; then $y=$ $1$ so that%
\[
\cos(\left(  g-1\right)  \arccos1)=\cos(\arccos1+c^{\prime})\text{.}%
\]

Now $\arccos1=2r\pi$, with $r=0,1,...$. Hence $\cos(h\arccos1)=1$, writing
$rh=r^{\prime}$, with $r^{\prime}\in%
\mathbb{Z}
$. It follows that $1=\cos(\arccos1+c^{\prime})$. Write $\arccos1+c^{\prime
}=\beta$; this implies $1=\cos\beta$, i.e. $\beta=2r^{\prime\prime}\pi$, with
$r^{\prime\prime}\in%
\mathbb{N}
$.

Note that
\[
\arccos1+c^{\prime}=2h^{\prime}\pi
\]
and therefore
\[
c^{\prime}=2h^{\prime}\pi-\arccos1=2h^{\prime}\pi-2m\pi=2\pi(h^{\prime
}-m)=2h^{\prime\prime}\pi\text{, with }h^{\prime\prime}\in%
\mathbb{Z}
\text{.}%
\]

For the constant $c^{\prime}$ we may therefore consider any multiple of $\pi$,
including $0$, i.e. $\cos(\left(  g-1\right)  \arccos\mathsf{v})=\cos(\arccos
y+c^{\prime})$ with $c^{\prime}=0$.

A solution of the initial differential equation is therefore given by
\[
y=\cos(\left(  g-1\right)  \arccos\mathsf{v})\text{, for }\mathsf{v}\in
\lbrack\mathsf{v}^{\ast},1]\text{.}%
\]

The polynomial $T_{g-1}$ increases from $\mathsf{v}^{\ast}$ to $1$, it should
decrease at the left of $\mathsf{v}^{\ast}$ . Define $\mathsf{v}^{\ast\ast}$
the point where it starts its decrease. The point $\mathsf{v}^{\ast\ast}$ is
therefore a maximum point with $y(\mathsf{v}^{\ast\ast})=1$ and $y$ decreases
on $[\mathsf{v}^{\ast\ast},\mathsf{v}^{\ast}].$ Therefore $y^{\prime}<0$ in $[\mathsf{v}^{\ast\ast},\mathsf{v}^{\ast}]$ and
\[
-\frac{y^{\prime}}{\sqrt{1-y^{2}}}=\frac{g-1}{\sqrt{1-\mathsf{v}^{2}}}\text{,}%
\]
since $\sqrt{(y^{\prime})^{2}}=|y^{\prime}|=-y^{\prime}$.

Therefore
\begin{align*}
\int-\frac{y^{\prime}}{\sqrt{1-y^{2}}}dy &  =\int\frac{g-1}{\sqrt
{1-\mathsf{v}^{2}}}d\mathsf{v}+c\text{, }\arccos y\\
&  =\left(  g-1\right)  \arccos\mathsf{v}+c,
\end{align*}
which together with a similar argument as previously yields to adopt $c=0$.

Since $\mathsf{v}^{\ast}$ coincides with $\mathsf{x}_{g-2}$, and
$\mathsf{v}^{\ast\ast}=x_{g-3}$, we iterate the above arguments for all nodes
until $x_{0}$ ; we conclude that $y=\cos(\left(  g-1\right)  \arccos
\mathsf{v})$.\\

Proceeding as above on any of the intervals between the nodes concludes the proof.
\end{proof}

We now obtain the roots $\widetilde{x}_{k}$ of the derivative of
$x\rightarrow T_{g-1}(x).$

\begin{proposition}
It holds
\begin{equation}
\widetilde{x}_{k}=\cos\left(  \frac{k\pi}{g-1}\right)  \text{, for
}k=1,...,g-2\text{.}\label{roots}%
\end{equation}
\end{proposition}

\begin{proof}
We get the roots $\widetilde{x}_{k}$ through a first order differentiation.
With $\theta=\left(  g-1\right)  \arccos\mathsf{v}$ it holds%
\[
y^{\prime}=T_{g-1}^{\prime}(\mathsf{v})=-\sin\theta\frac{d\theta}{d\mathsf{v}%
}=\left(  \sin(h\arccos\mathsf{v})\right)  \frac{g-1}{\sqrt{1-\mathsf{v}^{2}}%
}\text{.}%
\]

Note that
\[
\frac{h}{\sqrt{1-\mathsf{v}^{2}}}\sin(h\arccos\mathsf{v})=0\text{,}%
\sin(h\arccos\mathsf{v})=0
\]
hence
\[
h\arccos\mathsf{v}=k\pi,\arccos\mathsf{v}=\frac{k\pi}{g-1},x_{k}=\cos\left(
\frac{k\pi}{g-1}\right)  .
\]

Then we get
\begin{align*}
y(\widetilde{x}_{k}) &  =\cos(\left(  g-1\right)  \arccos x_{k})\\
&  =\cos\left(  \left(  g-1\right)  \frac{k\pi}{g-1}\right)  =(-1)^{k}\text{,}%
\end{align*}
which yields%
\[
\widetilde{x}_{k}=\cos\left(  \frac{k\pi}{g-1}\right)  \text{, for
}k=1,...,g-2.
\]
We rename the $\widetilde{x}_{k}$'s in increasing order.
\end{proof}

\bigskip

Note that taking $k=0$ and $k=g-1$ in (\ref{roots}) we recover $T_{g-1}%
(-1)=T_{g-1}(1)=1$ so that all $g$ points $\widetilde{x}_{k}$, $0\leq k\leq
g-1$ are points of maximal or minimal value of $T_{g-1}$ in $\left[
-1,1\right]  .$\\

The optimal design is therefore given by
\[
\mathcal{\xi}^{\ast}:=\left\{  \left(  \left[  n\frac{\left\vert l_{k}\left(
x\right)  \right\vert }{\sum_{j=0}^{g-1}\left\vert l_{j}\left(  x\right)
\right\vert }\right]  ;\cos\left(  \frac{k\pi}{g-1}\right)  \right)  \text{,
for }k=0,...,g-1\right\}  .
\]

\begin{definition}
The nodes
\[
\widetilde{x}_{k}=\cos\left(  \frac{k\pi}{g-1}\right)  \text{, for
}k=0,...,g-1.
\]
are the Chebyshev nodes on $\left[  -1,1\right]  .$
\end{definition}

\begin{example}
(Hoel - Levine) \ Consider $%
\mathbb{R}
\rightarrow%
\mathbb{R}
,x\mapsto P_{3}\left(  x\right)  :=\sum_{j=0}^{3}\theta_{j}x^{j}$, $\theta
_{j}$, $j=0;1;2;3.$ We intend to estimate $P_{3}\left(  2\right)  $ using
$n=52 $ observations in $\left[  -1;1\right]  .$ The model writes as
$y_{i}\left(  x_{j}\right)  :=P_{3}\left(  x\right)  +\varepsilon_{i}$ with
$\varepsilon_{i}\thicksim N\left(  0;1\right)  $ i.i.d, for all $i$ and $j$,
$i=1,...,n_{j} $, $j=0,..,3,$ $\sum_{j=0}^{3}n_{j}=52.$The optimal design
$\overline{\xi}^{\ast}$ is given by:
\[
\left(  x_{j}=\cos\left(  \frac{j\pi}{h}\right)  ,\left[  n\frac{\left\vert
l_{j}\left(  x\right)  \right\vert }{\sum_{j=0}^{g-1}\left\vert l_{j}\left(
x\right)  \right\vert }\right]  \right)  .
\]
Therefore $x_{0}=-1,x_{1}=-\frac{1}{2},x_{2}=\frac{1}{2},x_{3}=1$ and
$\overline{\xi}^{\ast}\left(  -1\right)  =\frac{5}{52},\overline{\xi^{\ast}%
}\left(  -\frac{1}{2}\right)  =\frac{12}{52},\overline{\xi}^{\ast}\left(
\frac{1}{2}\right)  =\frac{20}{52},\overline{\xi}^{\ast}\left(  1\right)
=\frac{15}{52}.$ Instead of this optimal design, consider the design with
nodes \textbf{supp}$\left(  \xi\right)  =\left\{  -1,\frac{-1}{3},\frac{1}%
{3},1\right\}  $ and weights $\xi\left(  -1\right)  =\xi\left(  \frac{-1}%
{3}\right)  =\xi\left(  \frac{1}{3}\right)  =\xi\left(  1\right)  =\frac
{13}{52}$, it holds $var_{\xi}\left(  \mathcal{L}_{3}(\widehat{P_{3}%
})(2)\right)  \thicksim20>var_{\xi^{\ast}}\left(  \mathcal{L}_{3}%
(\widehat{P_{3}})(2)\right)  \thicksim13$ where $var_{\xi}$ and
$var_{\overline{\xi^{\ast}}}$ denote respectively the variance under the
corresponding design.
\end{example}

\subsection{\bigskip Hoel Levine optimal design and the uniform approximation
of functions}

We now make some comment pertaining to the polynomial $x\rightarrow
T_{g-1}(x)$ in relation with the theory of the best uniform polynomial
approximation of functions on $\left[  -1,1\right]  .$

Consider the monomial $x\rightarrow x^{g-1}$ on $\left[  -1,1\right]  $ and
define its best uniform $q_{g-2}$ approximation in the linear space of all
polynomials with degree less or equal $g-2.$ By Theorem
\ref{ThmBorel-Tchebicheff} there exist $g$ equi-oscillation points
$z_{0},..,z_{g-1}$ in $\left[  -1,1\right]  $ where
\[
z_{k}^{g-1}-q_{g-2}(z_{k})=\left(  -1\right)  ^{k}\sup_{x\in\left[
-1,1\right]  }\left\vert x^{g-1}-q_{g-2}(x)\right\vert .
\]
Some analysis proves that the coefficient of $x^{g-1}$ in $T_{g-1}$ is
$C:=1/2^{g-2}.$

Observe that the polynomial $x\rightarrow\widetilde{T}_{g-1}(x):=$
$CT_{g-1}(x)$ shares the same oscillation properties and degree as
$x\rightarrow$ $x^{g-1}-q_{g-2}(x)$ on $\left[  -1,1\right]  ,$but for its
extreme values$.$ Now by Theorem \ref{ThmBorel-Tchebicheff} those properties
identify $T_{g-1}$ with \ the error function $x\rightarrow x^{g-1}-q_{g-2}(x)
$ up to a multiplicative constant $.$

Consider now the null function $x\rightarrow0(x)$ defined on $\left[
-1,1\right]  $ and let $p_{g-1}$ be a generic polynomial defined on $\left[
-1,1\right]  $ with degree $g-1$ and whose coefficient of $x^{g-1}$ is $1.$
Then clearly
\[
\sup_{x\in\left[  -1,1\right]  }\left\vert \widetilde{T}_{g-1}%
(x)-0(x)\right\vert \leq\sup_{x\in\left[  -1,1\right]  }\left\vert
p_{g-1}(x)-0(x)\right\vert .
\]
Note that in the present case when $x\rightarrow0(x)$ is to be approximated,
uniqueness of the best uniform approximating function is defined up to a
multiplicative constant; therefore we may say that $T_{g-1}$ is, up to the
multiplicative constant $C$ the best uniform approximation of $x\rightarrow
0(x)$ by polynomials with degree $g-1$ in $\left[  -1,1\right]  .$

This fact is the entry point to more general optimal designs when estimating
functions outside the context of polynomials; see \cite{Karlin1966a}.

\section{Uniform interpolation optimal designs (Guest)}

Consider the uniform variance of the estimator (with respect to $x$). A
natural strong criterion for optimality is defined through
\begin{equation}
\min_{\left\{  n_{j}\in%
\mathbb{N}
^{\ast},\text{ }j=0,...,g-1:\sum_{j=0}^{g-1}n_{j}=n\right\}  }\max
_{x\in\left[  -1;1\right]  }var\left(  \mathcal{L}_{n}(\widehat{f})(x)\right)
.\label{CritereGuest}%
\end{equation}

In this section two goals will be reached. First we obtain the optimal design
$\xi^{\ast}$ solving (\ref{CritereGuest}). Then we will show that
extrapolation designs are of a different nature with respect to interpolation
ones, since, as seen below,%
\[
var_{\mathcal{\xi}^{\ast}}\left(  \mathcal{L}_{n}(\widehat{f})(x)\right)
\neq\min_{\xi}var\left(  \mathcal{L}_{n}(\widehat{f})(x)\right)  \text{, for
}x>1\text{ }%
\]
where the minimum upon $\xi$ all designs depends on $x.$ Here we consider an
extrapolation design with $x>1.$\\

Define the Legendre polynomials on $\left[  -1,1\right]  .$

\begin{definition}
The Legendre polynomial of order $g-1$ on $\left[  -1,1\right]  $ is defined
by
\begin{align*}
P_{g-1}\left(  x\right)   &  :=\frac{1}{2^{g-1}\left(  g-1\right)  !}%
\frac{d^{g-1}}{dx^{g-1}}\left(  x^{2}-1\right)  ^{g-1}\\
&  =2^{-\left(  g-1\right)  }\sum_{j=0}^{\left[  \frac{g-1}{2}\right]
}\left(  -1\right)  ^{g-1}\left(
\begin{tabular}
[c]{l}%
$g-1$\\
$j$%
\end{tabular}
\ \right)  \left(
\begin{tabular}
[c]{l}%
$2\left(  g-1-j\right)  $\\
$g-1$%
\end{tabular}
\ \right)  x^{g-1-2j}.
\end{align*}

\end{definition}

\bigskip

\begin{remark}
The relation $P_{g-1}\left(  x\right)  :=\frac{1}{2^{g-1}\left(  g-1\right)
!}\frac{d^{g-1}}{dx^{g-1}}\left(  \left(  x^{2}-1\right)  ^{g-1}\right)  $ is
known as Rodriguez formula.
\end{remark}

\begin{remark}
Clearly $P_{g-1}$ has $g-1$ roots in $\left(  -1,1\right)  ,$ as seen now.
Indeed the polynomial $\left(  x^{2}-1\right)  ^{g-1}$ has degree $2\left(
g-1\right)  $, and it has multiple roots at points $\pm1$. By Rolle's Theorem
its derivative admits a root inside $\left(  -1,1\right)  .$ This derivative
assumes also the value $0$ at $\pm1$, since it has at least three roots in
$\left[  -1,1\right]  $. Apply once more Rolle's theorem to the second
derivative, which takes value $0$ at $\pm1$, since it has at least four roots.
Proceeding further , the $\left(  g-1\right)  -$th derivative has $g-1$ roots
in $\left(  -1,1\right)  .$ Up to a constant this derivative is the Legendre
polynomial $P_{g-1}.$
\end{remark}

\begin{remark}
\label{remarkPoints extremaux-1+1}The value of $P_{g-1}\left(  x\right)  $ at
$x=\pm1$ can be obtained. Indeed it holds
\[
\left(  x^{2}-1\right)  ^{g-1}=\left(  x-1\right)  ^{g-1}\left(  x+1\right)
^{g-1}.
\]
By Leibnitz formula\bigskip\
\begin{align*}
&  \frac{d^{g-1}\left(  \left(  x-1\right)  ^{g-1}\left(  x+1\right)
^{g-1}\right)  }{dx^{g-1}}\\
&  =\sum_{j=0}^{g-1}\left(
\begin{tabular}
[c]{l}%
$g-1$\\
$j$%
\end{tabular}
\ \right)  \frac{d^{j}\left(  \left(  x-1\right)  ^{g-1}\right)  }{dx^{j}%
}\frac{d^{g-1-j}\left(  x+1\right)  ^{g-1}}{dx^{g-1-j}}.
\end{align*}
For $j=0,...,g-2,$ it holds
\[
\left(  \frac{d^{j}\left(  \left(  x-1\right)  ^{g-1}\right)  }{dx^{j}%
}\right)  _{x=1}=0
\]
and
\[
\left(  \frac{d^{g-1}\left(  x-1\right)  ^{g-1}}{dx^{g-1}}\right)
_{x=1}=\left(  g-1\right)  !.
\]
Henceforth
\[
\frac{d^{g-1}\left(  \left(  x-1\right)  ^{g-1}\left(  x+1\right)
^{g-1}\right)  }{dx^{g-1}}=\left(  g-1\right)  !2^{g-1}.
\]
This yields
\begin{equation}
P_{g-1}\left(  1\right)  =1\text{ and }P_{g-1}\left(  -1\right)  =\left(
-1\right)  ^{g-1}.\label{(2)}%
\end{equation}

\end{remark}

\bigskip

\bigskip We need some facts about the Lagrange elementary polynomials;
denoting
\[
\pi\left(  x\right)  :=%
{\textstyle\prod\limits_{j=0}^{g-1}}
\left(  x-x_{j}\right)
\]
it holds

\bigskip

\begin{lemma}
\label{LemmaGuest1}It holds (i)%
\begin{align*}
\left(  \frac{d\pi\left(  x\right)  }{dx}\right)  _{x=x_{j}} &  =0\text{ for
}j=1,...,g-2,\text{ \ }\\
\text{ iff }\left(  \frac{d^{2}\pi\left(  x\right)  }{dx^{2}}\right)
_{x=x_{j}} &  =0\text{ for }j=1,...,g-2.
\end{align*}
(ii) $\pi\left(  x\right)  =\alpha\left(  x^{2}-1\right)  \phi_{g-2}\left(
x\right)  ,$with%
\[
\phi_{g-2}\left(  x\right)  =\frac{dP_{g-1}\left(  x\right)  }{dx}%
\]
\bigskip where $P_{g-1}$ is the Legendre polynomial of order $g-1$ on $\left[
-1,1\right]  .$

Finally (iii)
\[
\text{ }\left(  \frac{d}{dx}_{j}l_{j}\left(  x\right)  \right)  _{x=x_{j}%
}=0\text{ iff }\left(  \frac{dP_{g-1}\left(  x\right)  }{dx}\right)
_{x=x_{j}}=0.
\]

\end{lemma}

\begin{proof}
Denoting%
\[
\pi\left(  x\right)  :=%
{\textstyle\prod\limits_{j=0}^{g-1}}
\left(  x-x_{j}\right)
\]
write%
\[
l_{j}\left(  x\right)  =\frac{\pi\left(  x\right)  }{\left(  x-x_{j}\right)
\left(  \frac{d\pi\left(  x\right)  }{dx}\right)  _{x_{j}}}.
\]
We have%
\[
\pi\left(  x\right)  =\left(  x-x_{j}\right)  \left(  \frac{d\pi\left(
x\right)  }{dx}\right)  _{x_{j}}l_{j}\left(  x\right)  ,
\]%
\[
\frac{d\pi\left(  x\right)  }{dx}=\left(  x-x_{j}\right)  \left(  \frac
{d\pi\left(  x\right)  }{dx}\right)  _{x_{j}}\frac{dl_{j}\left(  x\right)
}{dx}+\left(  \frac{d\pi\left(  x\right)  }{dx}\right)  _{x_{j}}l_{j}\left(
x\right)
\]
and%
\[
\frac{d^{2}\pi\left(  x\right)  }{dx^{2}}=\left(  \frac{d\pi\left(  x\right)
}{dx}\right)  _{x_{j}}\left\{  \left(  x-x_{j}\right)  \frac{d^{2}l_{j}\left(
x\right)  }{dx^{2}}+2\frac{dl_{j}\left(  x\right)  }{dx}\right\}  .
\]
This last display proves (i).

In%
\[
\pi\left(  x\right)  :=%
{\textstyle\prod\limits_{j=0}^{g-1}}
\left(  x-x_{j}\right)
\]
the $x_{j}$'s, $j=0,...,g-1$, are the abscissas where the variance function is
minimax. Indeed in (\ref{(3)}) we have proved that the absolute value of the
elementary Lagrange polynomial takes value $1$ , which is its maximal value,
when evaluated on the nodes. Hence the variance
\[
var\left(  \mathcal{L}_{n}(\widehat{f})(x)\right)  =\sum_{i=0}^{g-1}\left(
l_{j}\left(  x\right)  \right)  ^{2}\frac{\sigma^{2}}{n_{i}}%
\]
takes its maximal values at points $x_{j}$'s.

Hence $\left\{  -1,1\right\}  \subset\left\{  x_{j},j=0,...,g-1\right\}  $ and
the remaining $g-2$ $x_{j}$'s are points of maximal value of the variance
inside $\left(  -1;1\right)  .$ Write the polynomial $\pi\left(  x\right)  $
as%
\[
\pi\left(  x\right)  =\alpha\left(  x^{2}-1\right)  \phi_{g-2}\left(
x\right)  ,
\]
where%
\[
\alpha\phi_{g-2}\left(  x\right)  :=%
{\textstyle\prod\limits_{j=0,\text{ }j\neq-1;1}^{g-1}}
\left(  x-x_{j}\right)  .
\]
The polynomial $\phi_{g-2}$, with degree $g-2$ is determined through the
conditions%
\[
\left(  \frac{d^{2}\pi\left(  x\right)  }{dx^{2}}\right)  _{x=x_{j}}=0\text{
for }j=1,...,g-2.
\]
Since
\[
\frac{d\pi\left(  x\right)  }{dx}=2\alpha x\phi_{g-2}\left(  x\right)
+\alpha\left(  x^{2}-1\right)  \frac{d\left(  \phi_{g-2}\left(  x\right)
\right)  }{dx}%
\]
and%
\[
\frac{d^{2}\pi\left(  x\right)  }{dx^{2}}=2\alpha\phi_{g-2}\left(  x\right)
+4\alpha x\frac{d\left(  \phi_{g-2}\left(  x\right)  \right)  }{dx}%
+\alpha\left(  x^{2}-1\right)  \frac{d^{2}\left(  \phi_{g-2}\left(  x\right)
\right)  }{dx^{2}}%
\]
those conditions amount to the system%
\[
\left\{
\begin{tabular}
[c]{l}%
$0=2\alpha\phi_{g-2}\left(  x_{1}\right)  +4\alpha x_{1}\left(  \frac{d\left(
\phi_{g-2}\left(  x\right)  \right)  }{dx}\right)  _{x=x_{1}}+\alpha\left(
x_{1}^{2}-1\right)  \left(  \frac{d^{2}\left(  \phi_{g-2}\left(  x\right)
\right)  }{dx^{2}}\right)  _{x=x_{1}}$\\
$........................................................................$\\
$0=2\alpha\phi_{g-2}\left(  x_{j}\right)  +4\alpha x_{j}\left(  \frac{d\left(
\phi_{g-2}\left(  x\right)  \right)  }{dx}\right)  _{x=x_{j}}+\alpha\left(
x_{j}^{2}-1\right)  \left(  \frac{d^{2}\left(  \phi_{g-2}\left(  x\right)
\right)  }{dx^{2}}\right)  _{x=x_{j}}$\\
$........................................................................$\\
$0=2\alpha\phi_{g-2}\left(  x_{g-2}\right)  +4\alpha x_{g-2}\left(
\frac{d\left(  \phi_{g-2}\left(  x\right)  \right)  }{dx}\right)  _{x=x_{g-2}%
}+\alpha\left(  x_{g-2}^{2}-1\right)  \left(  \frac{d^{2}\left(  \phi
_{g-2}\left(  x\right)  \right)  }{dx^{2}}\right)  _{x_{g-2}}$%
\end{tabular}
\ \ \right.  .
\]
Now the derivative of the Legendre $\ $polynomial $P_{g-1}$ is precisely the
solution of this system (see \cite{Guest1958}). Hence%

\[
\phi_{g-2}\left(  x\right)  =\frac{dP_{g-1}\left(  x\right)  }{dx}.
\]
This closes the proof of (ii).

We prove (iii). It holds%
\[
l_{j}\left(  x\right)  =\frac{\pi\left(  x\right)  }{\left(  x-x_{j}\right)
\left(  \frac{d\pi\left(  x\right)  }{dx}\right)  _{x_{j}}}.
\]
By
\[
\frac{d}{dx}l_{j}(x)=\frac{\alpha(x^{2}-1)\frac{d^{2}}{dx^{2}}P_{g-1}%
(x)+2\alpha x\frac{d}{dx}P_{g-1}(x)}{\left(  x-\widetilde{x_{j}}\right)
K}-\frac{\alpha(x^{2}-1)\frac{d}{dx}P_{g-1}(x)}{\left(  x-\widetilde{x_{j}%
}\right)  ^{2}K}%
\]
for some constant $K.$ When $x=\widetilde{x_{j}}$ then (iii) follows.
\end{proof}

We now obtain the optimal design.\ It holds

\begin{proposition}
\label{PropGuest}The nodes of the optimal design $\xi^{\ast}$ are the $g-2$
solutions of the equation%
\[
\frac{d}{dx}P_{g-1}\left(  x\right)  =0
\]
and $-1,1.$ The optimal frequencies are defined by the relation
\[
n_{j}=\frac{g}{n}.
\]

\end{proposition}

\begin{proof}
Keeping the notation $l_{j}$ defined in (\ref{PolElemLagrange}), we have
\[
var\left(  \mathcal{L}_{n}(\widehat{f})(x)\right)  =\sum_{j=0}^{g-1}l_{j}%
^{2}\left(  x\right)  \frac{\sigma^{2}}{n_{j}}.
\]
Since $\frac{\sigma^{2}}{n_{j}}>0$, any $\frac{\sigma^{2}}{n_{j}}$ should be
minimal in order to make the sum minimal. Hence $\left(  n_{0}^{\ast
},...,n_{g-1}^{\ast}\right)  $ should solve
\begin{equation}
\left\{
\begin{tabular}
[c]{l}%
$min_{\left(  n_{0},...,n_{g-1}\right)  }\left(  \frac{1}{n_{0}}+...+\frac
{1}{n_{g-1}}\right)  $\\
$\sum_{j=0}^{g-1}n_{j}=n$%
\end{tabular}
\ \ \ \ \right.  .\label{PbGuest1}%
\end{equation}
Hence%
\[
n_{j}^{\ast}=\frac{g}{n}.
\]
The polynomial $\sum_{j=0}^{g-1}l_{x_{j}}^{2}\left(  x\right)  \frac
{\sigma^{2}}{n_{j}}$ has degree $2g-2$, and indeed has $g-1$ roots of order 2.
This\ function is a decreasing function of $x$ on $\left(  -\infty,-1\right)
$ and an increasing function of $x$ on $\left(  1,+\infty\right)  ;$ the
points $-1$ and $1$ are therefore points of local maximal value of the
variance. The variance has therefore $g$ local extrema in $\left[
-1,1\right]  .$ Hence there exist $g-2$ local extrema for the variance inside
$\left(  -1,1\right)  ;$ they lie between the roots of $var\left(
\mathcal{L}_{n}(\widehat{f})(x)\right)  .$ These extrema are maxima, since the
variance is a sum of squares and takes value $0$ $g-2$ times.

We suppose that we have those points at hand; call them $\widetilde{x}_{j}$,
$j=0,...,g-1.$On those points $\widetilde{x}_{j}$ the function $var\left(
\mathcal{L}_{n}(\widehat{f})(x)\right)  $ takes the value
\[
\sum_{i=0}^{g-1}l_{i}^{2}\left(  \widetilde{x}_{j}\right)  \frac{\sigma^{2}%
}{n_{i}}=\frac{\sigma^{2}}{n_{j}}%
\]
with
\[
l_{i}\left(  x\right)  :=%
{\textstyle\prod\limits_{j=0,j\neq i}^{g-1}}
\frac{x-\widetilde{x}_{j}}{\widetilde{x}_{i}-\widetilde{x}_{j}}.
\]

The function $x\rightarrow$ $l_{j}^{2}\left(  x\right)  $ takes its maximal
value for $x=\widetilde{x}_{j}$, with $l_{j}^{2}\left(  \widetilde{x}%
_{j}\right)  =1$ independently on $j.$ Therefore it holds%
\[
\max_{x\in\left[  -1;1\right]  }var\left(  \mathcal{L}_{n}(\widehat{f}%
)(x)\right)  =\max_{x\in\left[  -1;1\right]  }\sum_{j=0}^{g-1}l_{j}^{2}\left(
x\right)  \frac{\sigma^{2}}{n_{j}}.
\]
The principle leading to the optimal design should now be made precise.\ The
largest variance of $\left(  \mathcal{L}_{n}(\widehat{f})(x)\right)  $ should
be attained on the points of measurements, in order to be able to control
it.\ Consider two nodes $\widetilde{x}_{i}$ and $\widetilde{x}_{k}.$ Then
\[
var\left(  \mathcal{L}_{n}(\widehat{f})(\widetilde{x}_{i})\right)  =var\left(
\mathcal{L}_{n}(\widehat{f})(\widetilde{x}_{k})\right)  =\frac{\sigma^{2}%
g^{2}}{n}.
\]
Hence
\[
\max_{x\in\left[  -1;1\right]  }var\left(  \mathcal{L}_{n}(\widehat{f}%
)(x)\right)  =\frac{\sigma^{2}g^{2}}{n}.
\]
The nodes should hence be the points of maximal value of the variance, which
equals $\frac{\sigma^{2}g^{2}}{n}.$

The first derivative of $var\left(  \mathcal{L}_{n}(\widehat{f})(x)\right)  $
writes%
\[
\frac{d}{dx}var\left(  \mathcal{L}_{n}(\widehat{f})(x)\right)  =\frac{2g}%
{n}\sum_{j=0}^{g-1}l_{j}\left(  x\right)  \frac{d}{dx}l_{j}\left(  x\right)  .
\]

It follows that finding the $g-2$ internal nodes $\widetilde{x}_{j}$'s results
in finding the solutions of the equation
\begin{equation}
\frac{d}{dx}var\left(  \mathcal{L}_{n}(\widehat{f})(x)\right)
_{x=\widetilde{x}_{j}}=0\label{(4)}%
\end{equation}
which by the above argument turns out to solve
\[
\left(  \sum_{j=0}^{g-1}l_{j}\left(  x\right)  \frac{d}{dx}l_{j}\left(
x\right)  \right)  _{x=\widetilde{x}_{j}}=0
\]
which yields, since
\[
l_{i}\left(  \widetilde{x}_{j}\right)  =\delta_{i,j}%
\]%
\[
\left(  \frac{d}{dx}l_{j}\left(  x\right)  \right)  _{x=\widetilde{x}_{j}%
}=0\text{ for all }j=1,..,g-2.
\]
This is a system of $g-2$ \ equations in the $\ g-2$ variables $\widetilde{x}%
_{1},..,\widetilde{x}_{g-2}.$ This system has precisely $g-2$ solutions,
solving (\ref{(4)}).

Apply Lemma \ref{LemmaGuest1} (iii) to conclude.
\end{proof}

\bigskip

We now characterize the performance of the optimal design $\xi^{\ast}$ through
an evaluation of the minimax variance (\ref{CritereGuest})
\[
\min_{\left\{  n_{j}\in%
\mathbb{N}
^{\ast},\text{ }j=0,...,g-1:\sum_{j=0}^{g-1}n_{j}=n\right\}  }\max
_{x\in\left[  -1;1\right]  }var\left(  \mathcal{L}_{n}(\widehat{f})(x)\right)
.
\]

\begin{lemma}
\label{LemmGuest2}\bigskip\ The Legendre polynomial $P_{g-1}$ is a solution of
the following differential equation (so-called Legendre equation)
\end{lemma}%

\[
\left(  1-x^{2}\right)  \frac{d^{2}f\left(  x\right)  }{dx^{2}}-2x\frac
{df\left(  x\right)  }{dx}+g\left(  g-1\right)  f\left(  x\right)  =0
\]
i.e%
\begin{equation}
\frac{d}{dx}\left(  \left(  x^{2}-1\right)  \frac{d}{dx}f(x)\right)
=g(g-1)f(x)\label{formLemmGuest2}%
\end{equation}

\begin{proof}
\ For an analytic function $f$ on $D$, by Cauchy formula, it holds%
\[
f^{\left(  g-1\right)  }\left(  x\right)  =\frac{\left(  g-1\right)  !}{2\pi
i}\int_{\gamma}\frac{f\left(  x\right)  }{\left(  z-x\right)  ^{g}}dz
\]
where $x$ is an interior point in $D$ and $\gamma$ is a regular circuit in $D
$ with $x$ in its interior. The variable $z$ runs on $\gamma$ in the positive
sense. Apply this formula to the analytic function
\[
f\left(  z\right)  =\left(  z^{2}-1\right)  ^{g-1},g=0,1,2,....
\]
By Rodriguez formula we obtain the following relation, known as Schl\"{a}fli
formula%
\[
P_{g-1}\left(  x\right)  =\frac{1}{2\pi i}\int_{\gamma}\frac{\left(
z^{2}-x\right)  ^{g-1}}{2^{g-1}\left(  z-x\right)  ^{g}}dz.
\]
Substituting now $f$ by $P_{g-1}$ in Legendre equation and applying the above
formula, we get
\begin{align*}
&  \left(  1-x^{2}\right)  \frac{d^{2}P_{g-1}}{dx^{2}}-2x\frac{dP_{g-1}}%
{dx}+g\left(  g-1\right)  P\\
&  =\frac{g}{2^{g-1}2\pi i}\int_{\gamma}\left(  \frac{d}{dz}\left(
\frac{\left(  z^{2}-1\right)  ^{g}}{\left(  z-x\right)  ^{g+1}}\right)
\right)  dz.
\end{align*}
Now
\[
\frac{g}{2^{g-1}2\pi i}\int_{\gamma}\left(  \frac{d}{dz}\left(  \frac{\left(
z^{2}-1\right)  ^{g}}{\left(  z-x\right)  ^{g+1}}\right)  \right)  dz=0.
\]
This can be written through
\[
\frac{d}{dx}\left(  \left(  x^{2}-1\right)  \frac{dP_{g-1}\left(  x\right)
}{dx}\right)  =g\left(  g-1\right)  P_{g-1}\left(  x\right)  .
\]
Indeed%
\[
\frac{d}{dx}\left(  \left(  x^{2}-1\right)  \frac{dP_{g-1}\left(  x\right)
}{dx}\right)  =2x\frac{dP_{g-1}\left(  x\right)  }{dx}+\left(  x^{2}-1\right)
\frac{d^{2}P_{g-1}}{dx^{2}}%
\]
and therefore%
\[
\frac{d}{dx}\left(  \left(  x^{2}-1\right)  \frac{dP_{g-1}\left(  x\right)
}{dx}\right)  =g\left(  g-1\right)  P_{g-1}\left(  x\right)
\]
which is%
\[
2x\frac{dP_{g-1}\left(  x\right)  }{dx}+\left(  x^{2}-1\right)  \frac
{d^{2}P_{g-1}}{dx^{2}}-g\left(  g-1\right)  P_{g-1}\left(  x\right)  =0
\]
which proves the claim.
\end{proof}

\bigskip We evaluate the local variance of the design of Guest for any point
$x$ in $\left[  -1,1\right]  $.

We now turn back to the points where the variance assumes its maximal values.
It holds%
\[
\pi\left(  x\right)  =\alpha\left(  x^{2}-1\right)  \frac{dP_{g-1}\left(
x\right)  }{dx}%
\]
hence%

\begin{align*}
\frac{d\pi\left(  x\right)  }{dx}  &  =\frac{d}{dx}\left(  \alpha\left(
x^{2}-1\right)  \frac{dP_{g-1}\left(  x\right)  }{dx}\right) \\
&  =\alpha g\left(  g-1\right)  P_{g-1}\left(  x\right)
\end{align*}
by Lemma \ref{LemmGuest2}.

Therefore%

\[
\frac{d^{2}\pi\left(  x\right)  }{dx^{2}}=\alpha g\left(  g-1\right)
\frac{dP_{g-1}\left(  x\right)  }{dx}.
\]

We evaluate the minimax variance, which we denote by $var_{\xi^{\ast}\text{ }%
}$.

It holds%

\begin{align*}
var_{\xi^{\ast}}\left(  \mathcal{L}_{n}(\widehat{f})(x)\right)   &
=\sum_{j=0}^{g-1}l_{j}^{2}\left(  x\right)  \frac{\sigma^{2}}{n_{j}}\\
&  =\sum_{j=0}^{g-1}\left(  \frac{\pi\left(  x\right)  }{\left(
x-x_{j}\right)  \left(  \frac{d\pi\left(  x\right)  }{dx}\right)  _{x_{j}}%
}\right)  ^{2}\frac{g\sigma^{2}}{n}\\
&  =\left(  \left(  x^{2}-1\right)  \frac{dP_{g-1}\left(  x\right)  }%
{dx}\right)  ^{2}\frac{\sigma^{2}}{g\left(  g-1\right)  ^{2}n}\sum_{j=0}%
^{g-1}\left(  \frac{1}{\left(  x-x_{j}\right)  P_{g-1}\left(  x_{j}\right)
}\right)  ^{2}.
\end{align*}

Making use of Lobatto formula and after some calculus (see
\cite{Hildebrand1956}) , we obtain%

\[
var_{\xi^{\ast}}\left(  \mathcal{L}_{n}(\widehat{f})(x)\right)  =\left(
1+\frac{x^{2}-1}{g\left(  g-1\right)  }\left(  \frac{d^{2}P_{g-1}\left(
x\right)  }{dx^{2}}\right)  ^{2}\right)  \frac{g\sigma^{2}}{n}.
\]

\bigskip As a consequence, using Guest minimax design, the maximal variance of
the interpolation is obtained at the boundaries $x=-1$ and $x=1.$ By symmetry
the minimal variance of the interpolation holds when $x=0.$

In the extrapolation zone, namely for large $\left\vert x\right\vert >1$

\bigskip%
\[
\frac{dP_{g-1}\left(  x\right)  }{dx}\thicksim\left(  g-1\right)
\frac{\left(  2\left(  g-1\right)  \right)  !}{2^{g-1}\left(  \left(
g-1\right)  !\right)  ^{2}}x^{g-2}.
\]
In the extrapolation zone this yields to the approximation%

\[
var_{\xi^{\ast}}\left(  \mathcal{L}_{n}(\widehat{f})(x)\right)  \thicksim
\left(  g-1\right)  \left(  \frac{\left(  2\left(  g-1\right)  !\right)
}{2^{g-1}\left(  \left(  g-1\right)  !\right)  ^{2}}\right)  ^{2}x^{2\left(
g-1\right)  }\frac{\sigma^{2}}{g-1}.
\]

\bigskip

Considering the points $x_{0}=-1,x_{g-1}=1$ which are also points of maximum
variance we see that the maximal variance cannot exceed $\frac{g}{n}\sigma
^{2}.$

\bigskip

\bigskip We have obtained the optimal minimax design in the interpolation range.

We now prove that this design is not suitable for extrapolation.

\section{The interplay between the Hoel-Levine and the Guest designs}

Without loss of generality we may consider the case when $c>1$; by
\[
var_{\xi}\left(  \mathcal{L}_{n}(\widehat{f})(x)\right)  =\sum_{j=0}%
^{g-1}l_{j}^{2}(x)\frac{\sigma^{2}}{n_{j}}%
\]
the variance of $\mathcal{L}_{n}(\widehat{f})(x)$, say $var_{\xi}\left(
\mathcal{L}_{n}(\widehat{f})(x)\right)  $ is an increasing function of $x$ for
$x>1$ for any design $\xi$ since \ the mapping $x\rightarrow l_{j}^{2}(x)$
increases for $x\geq1.$ It follows that for any $c>1$ the Hoel Levine design
$\xi_{c}$ is the minimax optimal extrapolation design on $\left(  1,c\right)
$ namely it solves%
\[
\min_{\xi\in\mathcal{M}_{1}^{\ast}}\max_{x\in(1,c]}var_{\xi}\left(
\mathcal{L}_{n}(\widehat{f})(x)\right)  .
\]
However there is no reason that $\xi_{c}$ be minimax optimal on whole $\left[
-1,c\right]  $ since it might not solve
\[
\min_{\xi\in\mathcal{M}_{1}^{\ast}}\max_{x\in\lbrack-1,c]}var_{\xi}\left(
\mathcal{L}_{n}(\widehat{f})(x)\right)  .
\]

\bigskip We consider the optimal minimax design on $\left[  -1,c\right]  $
with $c>1$ and discuss its existence and properties.

On $\left[  -1,1\right]  $ the optimal minimax design is Guest's design. We
will prove (see Proposition \ref{PropGuestnonOpt} hereunder) that this design
is not minimax optimal on $\left[  -1,c\right]  $ with $c>1$ for large $c.$

At the contrary we prove (Proposition \ref{PropHoel Levine sur (1,c)})
hereunder that the Hoel Levine design $\xi_{c}$ is minimax optimal on $\left[
1,c\right]  .$

Finally we prove (Proposition \ref{PropHLOpt}) that there exists a unique
$c^{\ast}>>1$ such that $\xi_{c^{\ast}\text{ }}$is minimax optimal on $\left[
-1,c^{\ast}\right]  .$

\begin{proposition}
\label{PropHoel Levine sur (1,c)}The Hoel Levine optimal design $\xi_{c}$ is
minimax optimal on $\left[  1,c\right]  $ for $c>1$ as proved in Section 2
(substitute $c<-1$ by $c>1$)$.$
\end{proposition}

\begin{proof}
This is a consequence of the fact that $x\rightarrow var_{\xi_{c}}\left(
\mathcal{L}_{n}(\widehat{f})(x)\right)  $ is an increasing function on
$\left[  1,c\right]  .$
\end{proof}

\begin{proposition}
\label{PropHLOpt}There exists $c_{1}>>1$ such that the Hoel Levine design
$\xi_{c^{\ast}}$ is minimax optimal on $[-1,c_{1}],$ i.e. it solves
\[
\min_{\xi\in\mathcal{M}_{1}^{\ast}}\max_{x\in\lbrack-1,c_{1}]}var_{\xi}\left(
\mathcal{L}_{n}(\widehat{f})(x)\right)  .
\]

\end{proposition}

\begin{proof}
We have seen that for $1<x<c$, the solution provided by Hoel and Levine is
minimax optimal. We now consider the case when $\left[  1,c\right]  $ is
substituted by $\left[  -1;c\right]  $ with $c>1.$

In this case the minimax optimal solution still holds as the Hoel - Levine
design if $c$ "large enough" .

Indeed let $var_{\eta}\left(  \mathcal{L}_{n}(\widehat{f})(x)\right)  $ be the
variance under a design $\eta$ whose support consists in the Chebyshev nodes
in $\left[  -1;1\right]  $. The design $\eta$ at this point is not defined in
a unique way, since the values of $\eta(x_{j})$ is not specified.

The function
\[
x\rightarrow var_{\eta}\left(  \mathcal{L}_{n}(\widehat{f})(x)\right)
\]
is continuous on $\left[  -1,1\right]  .$ Denote
\[
v_{\eta}^{\ast}:=\max_{x\in\left[  -1,1\right]  }var_{\eta}\left(
\mathcal{L}_{n}(\widehat{f})(x)\right)  .
\]
Assume that there exists some $c>1$ which does not depend on $\eta$ such that
\begin{equation}
v_{\eta}^{\ast}<var_{\eta}\left(  \mathcal{L}_{n}(\widehat{f})(c)\right)
.\label{cDans HL}%
\end{equation}
In such case it holds
\[
\min_{\eta\in\mathcal{M}_{\left[  -1,1\right]  }^{\ast}}v_{\eta}^{\ast}%
<\min_{\eta\in\mathcal{M}_{\left[  -1,1\right]  }^{\ast}}var_{\eta}\left(
\mathcal{L}_{n}(\widehat{f})(c)\right)  .
\]
The minimizing measure on the right hand side of the above display is
precisely the extrapolation Hoel Levine design at $c$ since the
function$~x\rightarrow var_{\eta}\left(  \mathcal{L}_{n}(\widehat{f}%
)(x)\right)  $ is increasing for $x>1.$

It remains to prove that such $c$ satisfying (\ref{cDans HL}) exists.\bigskip

For a given $c$ let%

\begin{align*}
R\left(  c\right)   &  :=\frac{\max_{\left[  -1;1\right]  }\left(  var_{\eta
}\left(  \mathcal{L}_{n}(\widehat{f})(x)\right)  \right)  }{var_{\eta}\left(
\mathcal{L}_{n}(\widehat{f})(c)\right)  }\\
&  =\frac{\max_{\left[  -1;1\right]  }\left(  \sum_{j=0}^{g-1}\frac{l_{j}%
^{2}\left(  x\right)  }{n_{j}}\right)  }{\sum_{j=0}^{g-1}\frac{l_{j}%
^{2}\left(  c\right)  }{n_{j}\left(  c\right)  }},
\end{align*}

with%

\[
n_{j}\left(  c\right)  :=\frac{\left\vert l_{j}\left(  c\right)  \right\vert
}{\sum_{i=0}^{g-1}\left\vert l_{i}\left(  c\right)  \right\vert }\text{ }%
\]
where the $n_{j}\left(  c\right)  ,0\leq j\leq g-1$ are the optimal
frequencies of the Hoel - Levine design evaluated in $x=c.$

We intend to prove that some $c>1$ exists for which $R(c)<1.$

If this holds then%

\begin{align*}
R\left(  c\right)   &  =\frac{\left(  \sum_{i=0}^{g-1}\left\vert l_{i}\left(
c\right)  \right\vert \right)  \left(  \max_{\left[  -1;1\right]  }\left(
\sum_{j=0}^{g-1}\frac{l_{j}^{2}\left(  x\right)  }{\left\vert l_{j}\left(
c\right)  \right\vert }\right)  \right)  }{\left(  \sum_{i=0}^{g-1}\left\vert
l_{i}\left(  c\right)  \right\vert \right)  \left(  \sum_{j=0}^{g-1}\left\vert
l_{j}^{2}\left(  c\right)  \right\vert \right)  }\\
&  =\frac{\left(  \max_{\left[  -1;1\right]  }\left(  \sum_{j=0}^{g-1}%
\frac{l_{j}^{2}\left(  x\right)  }{\left\vert l_{j}\left(  c\right)
\right\vert }\right)  \right)  }{\left(  \sum_{j=0}^{g-1}\left\vert
l_{j}\left(  c\right)  \right\vert \right)  }.
\end{align*}

Any of the $\left\vert l_{j}\left(  c\right)  \right\vert $ is an increasing
function of $c$ for $c>1;$ therefore $R\left(  c\right)  $ is a decreasing
function of $c$ for $c>1.$ Since each $l_{j}\left(  c\right)  \rightarrow
\infty$ as $c\rightarrow\infty,$ $R\left(  c\right)  $ will approach $0$ as
$c\rightarrow\infty.$

Since $l_{j}\left(  c\right)  \rightarrow0$ as $c\rightarrow1,$ for all $j,$
$R\left(  c\right)  $ will become infinite as $c\rightarrow1$. But $R\left(
c\right)  $ is a continuous function of $c$ for $c>1;$ consequently there will
exist a unique value of $c$, denote by $c_{1}$, satisfying $R\left(
c_{1}\right)  =1.$ For $c>c_{1},$ $R\left(  c\right)  <1;$ this entails that
$c$ exists with (\ref{cDans HL}).

The proof of Proposition \ref{PropHLOpt} is completed.
\end{proof}

\begin{remark}
The analytic derivation of $c_{1}$ is presented in \cite{Levine1966}.
\end{remark}

It follows from the same type of arguments as that just used to reject the
possibility of a Legendre (or Guest) design for $c>1$ that the Hoel - Levine
design cannot be optimum for $c<c_{1}.$ From continuity considerations one
would expect the optimum design to gradually change from the Guest spacing and
weighting to the Hoel - Levine spacing and weighting as $c$ increases from $1$
to $c_{1}.$\bigskip\ This is still an open question.

\begin{proposition}
\label{PropGuestnonOpt} The Guest design $\xi^{\ast}$ is not minimax optimal
on $\left[  -1,c\right]  $ for any $c>1$, which is to say that it not an
optimal extrapolating design.
\end{proposition}

\begin{proof}
By Proposition \ref{PropHoel Levine sur (1,c)} the Hoel Levine design on
$\left[  1,c\right]  $ is minimax optimal for large $c>1.$ By uniqueness of
the optimal design, following from the optimization problem, we deduce that
Guest design cannot coincide with this design.
\end{proof}

\section{ Confidence bound for interpolation/extrapolation designs}

Using a minimax optimal design we may produce a confidence bound for $f(x)$ at
any point $x$ in $\left[  -1,1\right]  $ or for $x$ far away from $\left[
-1,1\right]  $ $.$ We thus consider two cases for the location of $x.$ When $x
$ belongs to $\left[  -1,1\right]  $ then the optimal design is the Guest one.
By Proposition \ref{PropHLOpt} the Hoel Levine design is minimax on $\left[
-1,c_{1}\right]  $ for large $c_{1}.$ The minimax variance on $\left[
-1,c_{1}\right]  $ is therefore the variance of $\widehat{f\left(
c_{1}\right)  }$ since $var\left(  \widehat{f\left(  x\right)  }\right)  $ is
an increasing function of the variable $x$ for $x>1.$

Write%

\[
f\left(  x\right)  =\sum_{j=0}^{g-1}l_{j}\left(  x\right)  f\left(
x_{j}\right)  =\mathbf{l}\left(  x\right)  \mathbf{f}\left(
\widetilde{\mathbf{x}}\right)
\]

where $\widetilde{\mathbf{x}}:=\left(  x_{0},...,x_{g-1}\right)  ^{\prime}$
are the Chebychev nodes , $\mathbf{l}\left(  x\right)  :=\left(  l_{0}\left(
x\right)  ,...,l_{g-1}\left(  x\right)  \right)  ^{\prime}$ and%

\[
\text{ }\mathbf{f}\left(  \widetilde{\mathbf{x}}\right)  :=\left(
\begin{tabular}
[c]{l}%
$f\left(  x_{0}\right)  $\\
$.$\\
$f\left(  x_{j}\right)  $\\
$.$\\
$f\left(  x_{g-1}\right)  $%
\end{tabular}
\right)  .
\]

Assume that $y_{i}\left(  x_{j}\right)  :=f\left(  x_{j}\right)
+\varepsilon_{i,j}$ with $\varepsilon_{i,j}\thicksim N\left(  0;1\right)  $
i.i.d, for all $i$ and $j$, $i=1,...,n_{j}$, $j=0,..,g-1,$ $\sum_{j=0}%
^{g-1}n_{j}=n,$ where the $n$ observations are measured on $\left[
-1;1\right]  .$ The pointwise unbiased estimator of $f\left(  c\right)  $,
$c>1,$ is given by
\[
\widehat{f\left(  c\right)  }:=\mathbf{l}\left(  x\right)  \widehat{\mathbf{f}%
\left(  \widetilde{\mathbf{x}}\right)  }%
\]
where
\[
\widehat{\mathbf{f}\left(  \widetilde{\mathbf{x}}\right)  }=\left(  \frac
{1}{n_{0}}\sum_{i=1}^{n_{0}}y_{i}\left(  x_{0}\right)  ,...,\frac{1}{n_{g-1}%
}\sum_{i=1}^{n_{g-1}}y_{i}\left(  x_{g-1}\right)  \right)  .
\]

Since the distribution of the $y_{i}\left(  x_{j}\right)  ^{\prime}s$ is
$N\left(  f\left(  x_{j}\right)  ,1\right)  $ for all $i=1,...,n_{j}$ and
every $j=0,...,g-1$, the variance of the estimator $\widehat{f\left(
c_{1}\right)  }$ is given by%

\begin{align*}
var\left(  \widehat{f\left(  c_{1}\right)  }\right)   &  =var\left(
\sum_{j=0}^{g-1}l_{j}\left(  c_{1}\right)  \frac{\sum_{i=1}^{n_{j}}%
y_{i}\left(  x_{j}\right)  }{n_{j}}\right) \\
&  =\sum_{j=0}^{g-1}\frac{\left(  l_{j}\left(  c_{1}\right)  \right)  ^{2}%
}{n_{j}}.
\end{align*}

where the $n_{j}$'s are the frequencies of the Hoel Levine design evaluated at
point $c_{1}$ (which is indeed the minimax optimal design on $\left[
-1,c_{1}\right]  $ as argued above)$.$ The confidence set for $f(c_{1})$ is
given by%

\[
C_{n}:=\left(  \mathbf{l}\left(  c_{1}\right)  \right)  ^{\prime
}\widehat{\mathbf{f}\left(  \widetilde{\mathbf{x}}\right)  }\pm\sqrt
{p_{\alpha}\sum_{j=0}^{g-1}\frac{\left(  l_{j}\left(  c_{1}\right)  \right)
^{2}}{n_{j}}}.
\]

where
\[
\Pr\left(  N(0,1)>p_{\alpha}\right)  =1-\alpha
\]
and $N(0,1)$ is a random variable distributed with a standard normal law. It
holds%
\[
\Pr\left(  C_{n}\ni f(c_{1})\right)  \geq1-\alpha.
\]

When the variance of the $\varepsilon_{i}$'s are unknown then it can be
approximated by
\[
s^{2}:=\frac{\sum_{j=0}^{g-1}\left(  n_{j}-1\right)  s_{j}^{2}}{n-g-2}%
\]
where
\[
s_{j}^{2}:=\frac{\sum_{i=1}^{n_{j}}\left(  y_{i}(x_{j})-\left(  \sum
_{i=1}^{n_{j}}y_{i}(x_{j})\right)  /n_{j}\right)  ^{2}}{n-g-2}.
\]

The confidence area for $f(x)$ becomes
\[
C_{n}:=\left(  \mathbf{l}\left(  x\right)  \right)  ^{\prime}%
\widehat{\mathbf{f}\left(  \widetilde{\mathbf{x}}\right)  }\pm\sqrt
{q_{\alpha/2}\sum_{j=0}^{g-1}\frac{\left(  l_{j}\left(  c_{1}\right)  \right)
^{2}}{n_{j}s_{j}}}%
\]
where
\[
\Pr\left(  \left\vert t_{g-2}\right\vert >q_{\alpha/2}\right)  =1-\alpha
\]
where $t_{g-2}$ is a Student r.v. with $g-2$ degrees of freedom.

\section{ An application of the Hoel - Levine design, a multivariate case}

\subsection{\bigskip Some examples}

The above discussion may be \ applied for more general situations including
the regression models. We refer to the location/scale models, which are of
broad interest. Let%

\[
Z=\frac{Y\left(  x\right)  -\mu\left(  x\right)  }{\sigma},
\]

$\left(  \sigma,\text{ }\mu\right)  \in%
\mathbb{R}
^{+}\times\mathbf{F}$, with $\left(  \sigma,\text{ }\mu\right)  $ unknown and
$\mathbf{F}$ a known class of functions. The scale parameter $\sigma$ is
constant w.r.t. $x$ and $Z$ is a r.v. which is absolutely continuous w.r.t.
the Lebesgue measure. Its distribution is assumed to be known and does not
depend on $x$.

Write%

\[
f\left(  x\right)  :=\mu\left(  x\right)  +\sigma E\left(  Z\right)  ,
\]

\[
\varepsilon:=\sigma Z-\sigma E\left(  Z\right)  ,
\]

and therefore write the location/scale model as%

\[
Y\left(  x\right)  =f\left(  x\right)  +\varepsilon\text{.}%
\]

We consider some examples.

\bigskip

\begin{example}
The importance of the Weibull distribution in Reliability is well known.
Denote $T$ a Weibull r.v. with distribution function%
\[
F\left(  t\right)  =1-\exp\left(  -\left(  \frac{t}{\mu\left(  x\right)
}\right)  ^{\beta}\right)  ,t\geq0\text{.}%
\]

\end{example}

It can be written into%

\[
\ln T=\ln\mu\left(  x\right)  +\frac{1}{\beta}\ln\left(  -\ln\left(
1-F\left(  T\right)  \right)  \right)  \text{,}%
\]

and therefore%

\[
Y\left(  x\right)  =\ln\mu\left(  x\right)  +\sigma Z.
\]

where we wrote%

\[
Y\left(  x\right)  :=\ln T\text{, }\sigma:=\frac{1}{\beta}\text{, \ }%
Z:=\ln\left(  -\ln\left(  1-F\left(  T\right)  \right)  \right)  \text{.}%
\]

The model is therefore%

\[
Z=\frac{Y\left(  x\right)  -\ln\mu\left(  x\right)  }{\sigma}.
\]

Observe that%

\[
\Pr\left(  Z>t\right)  =e^{-e^{t}},t>0.
\]

Thus $Z$ is the Gumbel standard r.v.

Write the above model defining%

\[
\varepsilon:=\sigma Z-\sigma E\left(  Z\right)  ,
\]

so that%

\[
Y\left(  x\right)  =f\left(  x\right)  +\varepsilon\text{,}%
\]

where%

\[
f\left(  x\right)  :=\ln\mu\left(  x\right)  +\sigma E\left(  Z\right)  .
\]

\begin{example}
For a Gaussian r.v. $X\thicksim N\left(  \mu\left(  x\right)  ,\sigma
^{2}\right)  $\bigskip, it holds
\end{example}%

\[
Z=\frac{X-\mu\left(  x\right)  }{\sigma}\thicksim N\left(  0,1\right)  .
\]

\begin{example}
A regression model is clearly of the preceding type.
\end{example}

\begin{example}
Assume that $T$ is logistic, i.e.
\end{example}%

\[
F\left(  t\right)  =1-\left(  1+\exp\left(  \frac{t-f\left(  x\right)  }%
{\beta}\right)  \right)  ^{-1}.
\]
\bigskip

When $\beta=1$, we may write%

\begin{align*}
1-F\left(  t\right)   &  =\frac{1}{1+\exp\left(  t-f\left(  x\right)  \right)
},\text{ }\\
1+\exp\left(  t-f\left(  x\right)  \right)   &  =\frac{1}{1-F\left(  t\right)
}\text{, }\exp\left(  t-f\left(  x\right)  \right)  =\frac{F\left(  t\right)
}{1-F\left(  t\right)  }.
\end{align*}

It is enough to state%

\[
Z:=\ln\frac{F\left(  t\right)  }{1-F\left(  t\right)  }\text{, \ }%
\varepsilon:=Z-E\left(  Z\right)  ,\text{ }Y:=T
\]

to get%

\[
Y=f\left(  v\right)  +E\left(  Z\right)  +\varepsilon\text{.}%
\]

\subsection{\bigskip Multivariate optimal designs; a special case}

We extend the results of the above sections to a bivariate setting in\ a
reliability context; extension to multivariate similar cases is straightforward.

We consider the extrapolation problem with two variables.

Let%

\[
f:%
\mathbb{R}
^{2}\rightarrow%
\mathbb{R}
\text{, }\mathbf{x:=(}x,y)\mapsto f\left(  x,y\right)  :=\sum_{i_{1}=0}%
^{g_{1}-1}\sum_{i_{2}=0}^{g_{2}-1}a_{i_{1}i_{2}}\text{ }x^{i_{1}}y^{i_{2}%
}\text{ , \ }a_{i_{1}i_{2}}\in%
\mathbb{R}
\text{,}%
\]

be a polynomial in the two variables $x,y$ with partial degrees $g_{i}-1$ ,
$i=1,2$ in the variables $x,y$. The polynomial $f$ has $M_{1}$ $:=g_{1}g_{2}$
unknown coefficients.

In order to determine these coefficients we observe $f$ on a finite set
$\mathcal{E}$ in $%
\mathbb{R}
^{2}$. The fact that $\mathcal{E}$ consists in $M_{1}$ distinct points in $%
\mathbb{R}
^{2}$ is not sufficient for the estimation of the coefficients; it is
necessary that these points do not belong to an algebraic curve (or algebraic
hypersurface in higher dimension) . Indeed the identification for a polynomial
with many variables usually does not have a unique solution.\ For example
consider $n$ points $\left\{  \left(  x_{i},y_{i}\right)
:i=0,...,n-1\right\}  \subset%
\mathbb{R}
^{2}$, together with $n$ known values of $f$ on those points $\left\{
f\left(  x_{i},y_{i}\right)  :i=0,...,n-1\right\}  ;$ then there may not exist
a unique polynomial $P\left(  x,y\right)  $, such that%

\[
f\left(  x_{i},y_{i}\right)  =P\left(  x_{i},y_{i}\right)  \text{, }\left(
x_{i},y_{i}\right)  \in\mathcal{E}\text{.}%
\]

Indeed it is enough to consider the case when the $n$ distinct points are on a
line in $%
\mathbb{R}
^{3}$. In this case there exists an infinite number of planes $z=ax+by+c$,
which contain the $n$ points $\left(  x_{i},y_{i}\right)  .$

We will therefore assume that the $M_{1}$ points\ which define
$\mathcal{E\subset%
\mathbb{R}
}^{2}$ do not belong to an algebraic curve. This implies existence and
uniqueness for a polynomial which coincides with $f$ on $\mathcal{E}$, with
partial degree $g_{1}-1$ with respect to $x$ and $g_{2}-1$ w.r.t. $y.$ Denote
$P_{\mathcal{E}}\left(  f\right)  \left(  .\right)  $ this polynomial.

\bigskip

It can be proved that $\left(  x,y\right)  \rightarrow P_{\mathcal{E}}\left(
f\right)  \left(  x,y\right)  $ satisfies%

\[
f(x,y)=P_{\mathcal{E}}\left(  f\right)  \left(  x,y\right)  =\sum_{\left(
x_{i},y_{i}\right)  \in\mathcal{E}}f\left(  x_{i},y_{i}\right)  Q_{i}\left(
\mathcal{E}\text{,}\left(  x,y\right)  \right)
\]

where the polynomials $Q_{i}\left(  \mathcal{E}\text{,}\left(  x,y\right)
\right)  $ do not depend on $f.$ Indeed we may make $Q_{i}\left(
\mathcal{E}\text{, }.\right)  $ explicit; see e.g. \cite{Johnson1982} p 248-251.

Consider $\mathcal{E}$ a finite subset of the compact set $S:=\times_{i=1}%
^{2}\left[  a_{i},b_{i}\right]  $. Let the points $\left(  x_{i},y_{i}\right)
$ in $\mathcal{E}$ be%

\[
\left(  x_{i_{j}},y_{i_{j}}\right)
\]

$i_{j}=0,...,$ $g_{j}-1$ and $j=1,2$.

Define the elementary Lagrange polynomial in two variables by%

\begin{equation}
l_{i_{1}i_{2}}\left(  x,y\right)  :=l_{i_{1}}\left(  x\right)  l_{i_{2}%
}\left(  y\right) \label{lagrange-produit}%
\end{equation}

where%

\[
l_{i_{1}}\left(  x\right)  :=\frac{%
{\textstyle\prod\limits_{h_{1}\neq i_{1}\text{, }h_{1}=0}^{g_{1}-1}}
\left(  x-x_{h_{1}}\right)  }{%
{\textstyle\prod\limits_{h_{1}\neq i_{1}\text{, }h_{1}=0}^{g_{1}-1}}
\left(  x_{i_{1}}-x_{h_{1}}\right)  },l_{i_{2}}\left(  z\right)  :=\frac{%
{\textstyle\prod\limits_{h_{2}\neq i_{2}\text{, }h_{2}=0}^{g_{4}-1}}
\left(  z-z_{h_{2}}\right)  }{%
{\textstyle\prod\limits_{h_{2}\neq i_{2}\text{, }h_{2}=0}^{g_{4}-1}}
\left(  z_{i_{2}}-z_{h_{2}}\right)  }%
\]

are the elementary Lagrange polynomials with respect to the coordinates $x$
and $y.$

Clearly%

\[
l_{i_{1}i_{2}}\left(  x,y\right)  =\left\{
\begin{tabular}
[c]{l}%
$1$ if $\left(  x,y\right)  =\left(  x_{i_{1}},y_{i_{2}}\right)  $\\
$0$ \ otherwise
\end{tabular}
\right.  .
\]

The set%

\[
\left\{  l_{i_{1}i_{2}}\left(  x,y\right)  :i_{j}=0,...,g_{j}-1,j=1,2\right\}
\]

is a basis for the linear space of all polynomials with partial degree with
respect to the coordinate $x_{j}\ $less or equal $g_{j}-1$, for $j=1,2.$

The Gram matrix associated with this basis%

\[
\left(  l_{i_{1}i_{2}}\left(  x,y\right)  \right)  _{i_{1}=0,...,g_{1}%
-1,.i_{2}=0,...,g_{2}-1}%
\]

is therefore invertible; by uniqueness we have
\[
Q_{i}\left(  \mathcal{E}\text{,}\left(  x,y\right)  \right)  =l_{i_{1}i_{2}%
i}\left(  x,y\right)  .
\]

Therefore%

\[
P_{\mathcal{E}}\left(  f\right)  \left(  x,y\right)  :=\mathcal{L}\left(
P_{\mathcal{E}}\left(  f\right)  \right)  \left(  x,y\right)
\]

where we wrote%

\begin{equation}
\mathcal{L}\left(  P_{\mathcal{E}}\left(  f\right)  \right)  \left(
x,y\right)  =\sum_{i_{1}=0}^{g_{1}-1}\sum_{i_{2}=0}^{g_{2}-1}f\left(
x_{i_{1}},y_{i_{2}}\right)  \text{ }l_{i_{1\text{ }}i_{2\text{ }}}\left(
x,y\right)  \text{.}\label{Lagrange bivarie}%
\end{equation}
The above formula (\ref{Lagrange bivarie}) holds true since the nodes $\left(
x_{i_{l}},y_{i_{k}}\right)  $ belong to a rectangle (see \cite{Johnson1982}).

The polynomial $\mathcal{L}\left(  P_{\mathcal{E}}\left(  f\right)  \right)
\left(  x,y\right)  $ is called the bivariate Lagrange\ polynomial. The points
in $\mathcal{E}$ are the nodes for the interpolation of $f.$

When $f$ \ is not a polynomial but merely a function defined on $S$, which can
be extended by continuity on an open set $\mathcal{O}$ which contains $S$,
then the Lagrange interpolation scheme may be used as an approximation scheme
on $\mathcal{O}$; see \cite{Coatmelec1966}.

By uniqueness we adopt the notation
\[
P_{\mathcal{E}}\left(  f\right)  \left(  x,y\right)  =\mathcal{L}\left(
P_{\mathcal{E}}\left(  f\right)  \right)  \left(  x,y\right)  =f\left(
x,y\right)  .
\]

We now assume the following model%

\[
Z:=\frac{Y\left(  x,y\right)  -f\left(  x,y\right)  }{\sigma}%
\]

where $Z$ is a r.v. totally known in distribution with finite expectation
$E(Z)$ and finite variance $\eta^{2};$ the scale parameter $\sigma>0$ is
unknown and does not depend on $x,y;$ the coefficients of $P\left(
x,y\right)  $, \ $a_{i_{1}i_{2}}\in%
\mathbb{R}
$, are also unknown. We will see that the optimal design does not depend on
the constants $\sigma^{2}$ nor $\eta^{2}$.

Denote%

\[
\varepsilon\left(  x,y\right)  :=\sigma Z-\sigma E\left(  Z\right)
\]
whose variance equals $\eta^{2}\sigma^{2}.$

It holds%

\[
Y\left(  x,y\right)  =f\left(  x,y\right)  +\sigma E\left(  Z\right)
+\varepsilon\left(  x,y\right)  .
\]

We assume further that $f\left(  x,y\right)  $ can be observed only on a
subset $S$ in $\mathbb{R}^{2}.$\ In the setting of accelerated runs, this
subset $S$ is the stressed domain; it will be assumed that it is a rectangle
$\left[  a_{1},b_{1}\right]  \times$ $\left[  a_{2},b_{2}\right]  $ in
$\mathbb{R}^{2}$, a choice which is achievable by the experimenter.\ This
shape allows for an important simplification for the definition and the
calculation of the optimal design.

The lexicographic order on $%
\mathbb{R}
^{2}$ is defined as follows; for $\left(  x,y\right)  $ and $\left(
z,t\right)  $ in $%
\mathbb{R}
^{2}\times%
\mathbb{R}
^{2},\left(  x,y\right)  $ $\mathbf{\lesssim}$ $\left(  z,t\right)  $ iff
$x\leq z,y\leq t$.

Denote$\mathbf{\ a:=}\left(  a_{1},a_{2}\right)  $ be the point in $%
\mathbb{R}
^{2}$ which describes the threshold between the standard operational values of
the environment and the stressed conditions. With respect to $\mathbf{a}$ the
stressed region is a rectangle $\left[  \mathbf{a,b}\right]  $ north-east with
respect to $\mathbf{a}$, with south-west corner at $\mathbf{a}$, whereas the
unstressed domain is the south west quadrant $U$ with north east corner at
$\mathbf{a.}$ We denote $\mathbf{u}$ a point in%

\[
U:=\left\{  \left(  x,y\right)  \mathbf{\in%
\mathbb{R}
}^{2}\mathbf{:}\left(  x,y\right)  \mathbf{\lesssim a}\right\}  .
\]

We intend to find an optimal design in order to estimate the value of the
polynomial $f$ at point $\mathbf{u}$, hence we look for $\mathcal{E}$ and for
the number of observations on any of the points in $\mathcal{E}$, in such a
way to make the variance of the estimate of $P\left(  \mathbf{u}\right)  $ minimal.

Let $\left(  x_{i_{1}},y_{i_{2}}\right)  \in S$ $\ $be a node, i.e. a stress
configuration. We consider now the set of trials under this
configuration.\ Denoting \underline{$i$}$:=(i_{1},i_{2})\in\times_{j=1}%
^{2}\left\{  0,...,g_{j}-1\right\}  $ , we define $n(\underline{i})$ be the
total number of replications of the measurement $Y$ at point $\left(
x_{i_{1}},y_{i_{2}}\right)  .$ We denote $\underline{\mathbf{Y}}\left(
\underline{i}\right)  $ the vector of these measurements; We assume that the
coordinates of $\underline{\mathbf{Y}}\left(  \underline{i}\right)  $ are
ordered; this is the common procedure when looking at lifetimes of a number
$n(\underline{i})$ of identical systems operating in parallel during the
trial. So $\underline{\mathbf{Y}}\left(  \underline{i}\right)  :=\left(
Y_{\left(  1\right)  }\left(  \underline{i}\right)  ,...,Y_{\left(  n\left(
\underline{i}\right)  \right)  }\left(  \underline{i}\right)  \right)  $ is an
ordered sample obtained from an i.i.d. sample with size $n(\underline{i}).$

\bigskip

The system of equations which represents the observations is therefore%

\[
\left(  1\right)  \left\{
\begin{tabular}
[c]{l}%
$y_{\left(  1\right)  }\left(  x_{i_{1}},y_{i_{2}}\right)  =f\left(  x_{i_{1}%
},y_{i_{2}}\right)  +\sigma E\left(  Z\right)  +\varepsilon_{1}\left(
x_{i_{1}},y_{i_{2}}\right)  $\\
$...........................................................................$%
\\
$y_{\left(  k\right)  }\left(  x_{i_{1}},y_{i_{2}}\right)  =f\left(  x_{i_{1}%
},y_{i_{2}}\right)  +\sigma E\left(  Z\right)  +\varepsilon_{k}\left(
x_{i_{1}},y_{i_{2}}\right)  $\\
$.............................................................................
$\\
$y_{\left(  n\left(  \underline{i}\right)  \right)  }\left(  x_{i_{1}%
},y_{i_{2}}\right)  =f\left(  x_{i_{1}},y_{i_{2}}\right)  +\sigma E\left(
Z\right)  +\varepsilon_{n\left(  \underline{i}\right)  }\left(  x_{i_{1}%
},y_{i_{2}}\right)  $%
\end{tabular}
\right.  .
\]

It holds%

\[
\sum_{i_{1}=0}^{g_{1}-1}\sum_{i_{2}=0}^{g_{2}-1}n\left(  \underline{i}\right)
=M_{1}%
\]

with $M_{1}\in%
\mathbb{N}
\diagdown\left\{  0\right\}  $ a fixed number.

We assume that the vectors of measurements $\underline{\mathbf{Y}}\left(
\underline{i}\right)  $ and $\underline{\mathbf{Y}}\left(  \underline{h}%
\right)  $ are independent for $\underline{i}\neq\underline{h}$.

The system $\left(  1\right)  $ may be written as%

\[
\underline{\mathbf{Y}}\left(  \underline{i}\right)  =X\theta+\varepsilon
\left(  \underline{i}\right)
\]

where $X$ is the matrix with $n\left(  \underline{i}\right)  $ rows and $2$ columns%

\[
X\left(  \underline{i}\right)  :=\left(
\begin{tabular}
[c]{ll}%
$1$ & $E\left(  Z\right)  $\\
$.$ & $.$\\
$1^{n\left(  \underline{i}\right)  }$ & $E\left(  Z\right)  $%
\end{tabular}
\right)
\]

and $\theta$ is a column vector with two rows%

\[
\theta:=\left(
\begin{tabular}
[c]{l}%
$f\left(  x_{i_{1}},y_{i_{2}}\right)  $\\
$\sigma$%
\end{tabular}
\right)  \text{.}%
\]

Finally
\[
\text{ }\underline{\varepsilon}\left(  \underline{i}\right)  :=\left(
\begin{array}
[c]{c}%
\varepsilon_{1}\left(  \underline{i}\right) \\
.\\
\varepsilon_{n\left(  \underline{i}\right)  }\left(  \underline{i}\right)
\end{array}
\right)  \text{, }\underline{\mathbf{Y}}\left(  \underline{i}\right)
:=\left(  \text{%
\begin{tabular}
[c]{l}%
$y_{\left(  1\right)  }\left(  x_{i_{1}},y_{i_{2}}\right)  $\\
$.$\\
$y_{\left(  n\left(  \underline{i}\right)  \right)  }\left(  x_{i_{1}%
},y_{i_{2}}\right)  $%
\end{tabular}
}\right)  \text{ }.
\]

Denote%

\[
\Omega^{-1}\left(  \underline{i}\right)  :=\left(  cov\left(  y_{\left(
a\right)  }\left(  \underline{i}\right)  ,y_{\left(  b\right)  }\left(
\underline{i}\right)  \right)  \right)  _{a,b}^{-1}%
\]
which we assume to exist for all $\underline{i}$. In the above display,
$\Omega^{-1}\left(  \underline{i}\right)  $ is a matrix of order
$n(\underline{i});$ the matrix $\Omega\left(  \underline{i}\right)  $ is not
the identity matrix of order $n\left(  \underline{i}\right)  $ since the
vector of observations is ordered according to its coordinates.
\[
\Omega^{-1}\left(  \underline{i}\right)  :=\left(
\begin{array}
[c]{ccc}%
\omega_{_{1,1}} & . & \omega_{1,n\left(  \underline{i}\right)  }\\
. & . & .\\
\omega_{n\left(  \underline{i}\right)  ,1} & . & \omega_{n\left(
\underline{i}\right)  ,n\left(  \underline{i}\right)  }%
\end{array}
\right)  .
\]

The expected value of the measurement $Y$ at point $\left(  x_{i_{1}}%
,y_{i_{2}}\right)  $ equals $f\left(  x_{i_{1}},y_{i_{2}}\right)  +\sigma
E\left(  Z\right)  .$ Denote $m_{i_{1},i_{2}\text{ }}$its GLS estimator

\bigskip%

\[
m_{\underline{i}}:=m_{i_{1},i_{2}\text{ }}=\left[  \left(  X^{\text{ }\prime
}\left(  \underline{i}\right)  \Omega^{-1}\left(  \underline{i}\right)
X\left(  \underline{i}\right)  \right)  \right]  ^{-1}X^{\text{ }\prime
}\left(  \underline{i}\right)  \Omega^{-1}\left(  \underline{i}\right)
\underline{\mathbf{Y}}\left(  \underline{i}\right)  .
\]

This estimator is strongly consistent.

We now define the estimator outside of the nodes.

Observe that for $\mathbf{u:=}\left(  x,y\right)  $%

\begin{align*}
\mathcal{L}\left(  P_{\mathcal{E}}\left(  f\right)  \right)  \left(
x,y\right)  +\sigma E\left(  Z\right)   & :=f\left(  x,y\right)  +\sigma
E\left(  Z\right) \\
& =\sum_{\left(  x_{i},y_{i}\right)  \in\mathcal{E}}\left(  f\left(
x_{i},y_{i}\right)  +\sigma E\left(  Z\right)  \right)  \text{ }l_{i_{1\text{
}}i_{2\text{ }}}\left(  x,y\right)  \text{.}%
\end{align*}

Denote $m(\mathbf{u)}$ the resulting estimator of $f\left(  \mathbf{u}\right)
+\sigma E\left(  Z\right)  $
\begin{align*}
m(\mathbf{u)}  & :=\sum_{\left(  x_{i},y_{i}\right)  \in\mathcal{E}%
}m_{\underline{i}}\text{ }l_{i_{1\text{ }}i_{2\text{ }}}\left(  \mathbf{u}%
\right) \\
& =\sum_{\left(  x_{i},y_{i}\right)  \in\mathcal{E}}m_{\underline{i}}\text{
}l_{i_{1\text{ }}}\left(  x\right)  l_{i_{2\text{ }}}\left(  y\right)  .
\end{align*}

This factorization relies on the fact that $S$ is a rectangle. We now evaluate
the variance of the unbiased estimator $m(\mathbf{u)}$; by independence of the
measurements on the nodes%

\begin{align*}
Var\left(  m(\mathbf{u)}\right)   & =\sum_{\underline{i}=\left(  i_{1}%
,i_{2}\right)  }^{2}\text{ }\left(  l_{i_{1\text{ }}}\left(  x\right)
l_{i_{2\text{ }}}\left(  y\right)  \right)  ^{2}var\left(  m_{\underline{i}%
}\right) \\
& =\sum_{\underline{i}=\left(  i_{1},i_{2}\right)  }\left(  l_{i_{1\text{ }}%
}\left(  x\right)  l_{i_{2\text{ }}}\left(  y\right)  \right)  ^{2}G\left(
\sigma^{2}\eta^{2},\Omega^{-1}\left(  \underline{i}\right)  ,X^{\text{ }%
\prime}\left(  \underline{i}\right)  \right)  ,
\end{align*}

where%

\begin{equation}
G\left(  \sigma^{2}\eta^{2},\Omega^{-1}\left(  \underline{i}\right)
,X^{\text{ }\prime}\left(  \underline{i}\right)  \right)  :=\sigma^{2}\eta
^{2}\left(  X^{\text{ }\prime}\left(  \underline{i}\right)  \Omega^{-1}\left(
\underline{i}\right)  X^{\text{ }}\left(  \underline{i}\right)  \right)
^{-1}X^{\text{ }\prime}\left(  \underline{i}\right)  .\label{G}%
\end{equation}

Note that $var\left(  \mathcal{L}\left(  P_{\mathcal{E}}\left(  f\right)
\right)  \left(  x,y\right)  \right)  \rightarrow0$, for $n\left(
\underline{i}\right)  \rightarrow\infty$, due to the fact that the generalized
lest-squares estimator is consistent under the present conditions.

The optimal design results as the solution to the following optimization problem,%

\[
\left\{
\begin{array}
[c]{c}%
\min\sum_{\underline{i}=\left(  i_{1},i_{2}\right)  }\left(  \text{
}l_{i_{1\text{ }}i_{2\text{ }}}\left(  x,y\right)  \right)  ^{2}G\left(
\sigma^{2}\eta^{2},\Omega^{-1}\left(  \underline{i}\right)  ,X^{\text{ }%
\prime}\left(  \underline{i}\right)  \right) \\
\mathbf{u}\in U\\
M_{1}=\sum_{\underline{i}}n\left(  \underline{i}\right)  .
\end{array}
\right.
\]

where the minimization is held on all choices of the set of measurements
(nodes) $\mathcal{E}$ and all frequencies $n\left(  \underline{i}\right)  .$

Although the problem generally has a numerical solution, in some practical
cases it is possible to obtain and analytic solution.

We explore a special case.

Define
\[
\Gamma:=\sum_{m,u}^{n\left(  \underline{i}\right)  }\omega_{m,u}\text{. }%
\]

Let $E\left(  Z\right)  =0$ and the distribution of $Y$ be symmetric around
$E\left(  Y(u)\right)  $.

In this case, $G\left(  \sigma^{2}\eta^{2},\Omega^{-1}\left(  \underline{i}%
\right)  ,X^{\text{ }\prime}\left(  \underline{i}\right)  \right)  $ becomes:%

\[
G_{1}:=G\left(  \sigma^{2}\eta^{2},\Omega^{-1}\left(  \underline{i}\right)
,X^{\text{ }\prime}\left(  \underline{i}\right)  \right)  =\frac{\sigma
^{2}\eta^{2}}{\Gamma}\text{ \ }%
\]

which depends on \underline{$i$} through $\Gamma;$ (See \cite{Lloyd1952} for
the proof).

In some cases $G_{1}$ may be simplified as follows,%

\begin{equation}
G_{2}:=g\left(  \sigma^{2}\eta^{2},\Omega^{-1}\left(  \underline{i}\right)
,X^{\text{ }\prime}\left(  \underline{i}\right)  \right)  =\frac{\sigma
^{2}\eta^{2}}{n\left(  \underline{i}\right)  }.\label{G2}%
\end{equation}

Indeed a necessary and sufficient condition for $G_{2}$ is%

\[
\left(
\begin{array}
[c]{c}%
1\\
.\\
1
\end{array}
\right)  ^{\prime}\Omega\left(  \underline{i}\right)  =\left(
\begin{array}
[c]{c}%
1\\
.\\
1
\end{array}
\right)  ^{\prime}\text{ .}%
\]
\ (see \cite{Downton1954}).

\bigskip In many cases the function $G$ takes on the form%

\[
\frac{\sigma^{2}\eta^{2}}{\alpha n\left(  \underline{i}\right)  +\beta}\left(
1+o\left(  1\right)  \right)
\]

where $\alpha$,$\beta$ are constants depending on the (known) distribution of
the random variable $Z,$ extending (\ref{G2}); see \cite{Celant2003}%
.{\LARGE \ }

The problem of determining the optimal design becomes%

\[
\left\{
\begin{array}
[c]{c}%
\min\sum_{\left(  i_{1},i_{2}\right)  }\frac{\left(  l_{\left(  i_{1}%
,i_{2}\right)  }\left(  \mathbf{u}\right)  \right)  ^{2}}{\alpha n\left(
\underline{i}\right)  +\beta}\\
M_{1}=\sum_{\underline{i}}n\left(  \underline{i}\right)  \text{, }n\left(
\underline{i}\right)  \in%
\mathbb{R}
^{+}.
\end{array}
\right.
\]

where the minimum holds on the choice of the nodes $\mathcal{E}$ and on the
frequencies. Fix $\left(  x_{i_{1}},y_{i_{2}}\right)  $ and apply the Theorem
of Karush-Kuhn-Tucker to%

\[
\left\{
\begin{array}
[c]{c}%
\min\sum_{\left(  i_{1},i_{2}\right)  }\frac{\left(  l_{\left(  i_{1}%
,i_{2}\right)  }\left(  \mathbf{u}\right)  \right)  ^{2}}{\alpha n\left(
\underline{i}\right)  +\beta}\\
M_{1}=\sum_{\underline{i}}n\left(  \underline{i}\right)  \text{, }n\left(
\underline{i}\right)  \in%
\mathbb{R}
^{+}%
\end{array}
\right.  .
\]

where the minimization is held on the frequencies $n\left(  \underline{i}%
\right)  .$ We obtain%
\[
\left[  n^{\ast}\left(  \underline{i}\right)  \right]  =\frac{\left\vert
l_{\left(  i_{1},i_{2}\right)  }\left(  \mathbf{u}\right)  \right\vert \left(
\alpha M_{1}+\beta%
{\textstyle\prod\limits_{j=1}^{2}}
g_{j}\right)  }{\sum_{\left(  i_{1},i_{2}\right)  =\underline{0}}\left\vert
l_{\left(  i_{1},i_{2}\right)  }\left(  \mathbf{u}\right)  \right\vert }%
-\beta\text{.}%
\]

Clearly, $n^{\ast}\left(  \underline{i}\right)  $ depends on $\ $the\ $\left(
x_{i_{1}},y_{i_{2}}\right)  $'s. We substitute $n^{\ast}\left(  \underline{i}%
\right)  $ in the variance formula, i.e. in%

\[
\sum_{\underline{i}=\underline{0}}^{\underline{l}}\left(  l_{\left(
i_{1},i_{2}\right)  }\left(  \mathbf{u}\right)  \right)  ^{2}\frac{\sigma
^{2}\eta^{2}}{\alpha n\left(  \underline{i}\right)  +\beta},
\]
with $l:=\left(  g_{1}-1\right)  \left(  g_{2}-1\right)  $ to obtain%

\[
Var\left(  m(\mathbf{u)}\right)  =\frac{\sigma^{2}\eta^{2}}{\alpha M_{1}%
+\beta\left(
{\textstyle\prod\limits_{j=1}^{2}}
g_{j}\right)  }\left(  \sum_{\underline{i}=\underline{0}}^{\underline{l}%
}\left\vert l_{\left(  i_{1},i_{2}\right)  }\left(  \mathbf{u}\right)
\right\vert \right)  ^{2}\text{.}%
\]
The optimal design hence will not depend on the value of $\sigma^{2}\eta^{2}$.
Optimizing with respect to the $\left(  x_{i_{1}},y_{i_{2}}\right)  ^{\prime
}s$ under the constraint $\mathbf{u}$ $\in U$, yields%

\[
\min_{\left\{  \left(  x_{i_{1}},y_{i_{2}}\right)  _{i_{1},i_{2}}\right\}
}\sum_{\underline{i}=\underline{0}}\left\vert l_{\left(  i_{1},i_{2}\right)
}\left(  \mathbf{u}\right)  \right\vert \text{.}%
\]

This is the same as the following two problems with one variable%

\[
\min\sum_{i_{j}=0}^{g_{j}-1}\left\vert l_{\left(  i_{1},i_{2}\right)  }\left(
\mathbf{u}\right)  \right\vert ,\text{ }j=1,2.
\]

With $\mathbf{a}:=(a_{1},a_{2}),$ the minimization is held when $j=1$ on the
abscissas $x_{j}$'s all larger than $a_{1}$ and smaller than $b_{1}$ and on
the ordinates $y_{j}$'s, all larger than $a_{2}$ and smaller than $b_{2}$ when
$j=2$, since%

\begin{align}
\sum_{\underline{i}=\underline{0}}\left\vert l_{\left(  i_{1,i_{2}}\right)
}\left(  \mathbf{u}\right)  \right\vert  & =\sum_{i_{1}=0}^{g_{1}-1}\left\vert
l_{i_{1}}\left(  u_{1}\right)  \right\vert \sum_{i_{k}=0}^{g_{2}-1}\left\vert
l_{i_{2}}\left(  u_{2}\right)  \right\vert \label{plan double}\\
\text{ with }\mathbf{u}\mathbf{:=}  & \left(  u_{1},u_{2}\right)  .\nonumber
\end{align}

Now minimizing the product in (\ref{plan double}) results in two independent
minimizations, one for each of the two factors, under the corresponding
constraint on the respective terms $u_{1}$ and $u_{2}.$ It follows that the
optimal design is the combination of two Hoel Levine marginal optimal designs.

Therefore the solution coincides with the previously obtained one, namely the
Hoel Levine design of Section 2, i.e.%

\[
s_{j}^{\ast}\left(  i_{j}\right)  =\frac{a_{j}+b_{j}}{2}+\frac{b_{j}-a_{j}}%
{2}\cos\left(  \frac{g_{j}-1-i_{j}}{g_{j}-1}\pi\right)  \text{, \ }%
i_{j}=0,...,g_{j}-1,j=1,2
\]
with $S:=\left[  a_{1},b_{1}\right]  \times\left[  a_{2},b_{2}\right]  .$

\bigskip
\bibliographystyle{amsplain}
\bibliography{biblio,biblioBIS,BIBLOIlIVRE}

\providecommand{\bysame}{\leavevmode\hbox to3em{\hrulefill}\thinspace}
\providecommand{\MR}{\relax\ifhmode\unskip\space\fi MR }
\providecommand{\MRhref}[2]{%
  \href{http://www.ams.org/mathscinet-getitem?mr=#1}{#2}
}
\providecommand{\href}[2]{#2}
\begin{thebibliography}{10}

\bibitem{Bazaraa2006}
Mokhtar~S. Bazaraa, Hanif~D. Sherali, and C.~M. Shetty, \emph{Nonlinear
  programming}, third ed., Wiley-Interscience [John Wiley \& Sons], Hoboken,
  NJ, 2006, Theory and algorithms. \MR{2218478 (2006k:90001)}

\bibitem{Celant2003}
Giorgio Celant, \emph{Extrapolation and optimal designs for accelerated runs},
  Ann. I.S.U.P. \textbf{47} (2003), no.~3, 51--84. \MR{2056619}

\bibitem{Coatmelec1966}
Christian Coatm{\'e}lec, \emph{Approximation et interpolation des fonctions
  diff\'erentiables de plusieurs variables}, Ann. Sci. \'Ecole Norm. Sup. (3)
  \textbf{83} (1966), 271--341. \MR{0232143 (38 \#469)}

\bibitem{Downton1954}
F.~Downton, \emph{Least-squares estimates using ordered observations}, Ann.
  Math. Statistics \textbf{25} (1954), 303--316. \MR{0061334 (15,810b)}

\bibitem{Dzyadyk2008}
Vladislav~K. Dzyadyk and Igor~A. Shevchuk, \emph{Theory of uniform
  approximation of functions by polynomials}, Walter de Gruyter GmbH \& Co. KG,
  Berlin, 2008, Translated from the Russian by Dmitry V. Malyshev, Peter V.
  Malyshev and Vladimir V. Gorunovich. \MR{2447076 (2009f:30001)}

\bibitem{Guest1958}
P.~G. Guest, \emph{The spacing of observations in polynomial regression}, Ann.
  Math. Statist. \textbf{29} (1958), 294--299. \MR{0094883 (20 \#1392)}

\bibitem{Hildebrand1956}
F.~B. Hildebrand, \emph{Introduction to numerical analysis}, McGraw-Hill Book
  Company, Inc., New York-Toronto-London, 1956. \MR{0075670 (17,788d)}

\bibitem{Hoel1964}
P.~G. Hoel and A.~Levine, \emph{Optimal spacing and weighting in polynomial
  prediction}, Ann. Math. Statist. \textbf{35} (1964), 1553--1560. \MR{0168102
  (29 \#5367)}

\bibitem{Johnson1982}
Lee~W. Johnson and Ronald~Dean Riess, \emph{Numerical analysis}, second ed.,
  Addison-Wesley Publishing Co., Reading, Mass., 1982. \MR{668699 (83j:65001)}

\bibitem{Karlin1966a}
Samuel Karlin and William~J. Studden, \emph{Optimal experimental designs}, Ann.
  Math. Statist. \textbf{37} (1966), 783--815. \MR{0196871 (33 \#5055)}

\bibitem{Kiefer1964}
J.~Kiefer and J.~Wolfowitz, \emph{Optimum extrapolation and interpolation
  designs. {I}, {II}}, Ann. Inst. Statist. Math. 16 (1964), 79--108; ibid.
  \textbf{16} (1964), 295--303. \MR{0178549 (31 \#2806)}

\bibitem{Kolmogorov1981}
A.~N. Kolmogorov and S.~V. Fomin, \emph{Elementy teorii funktsii i
  funktsionalnogo analiza}, fifth ed., ``Nauka'', Moscow, 1981, With a
  supplement ``Banach algebras'' by V. M. Tikhomirov. \MR{630899 (83a:46001)}

\bibitem{Levine1966}
A.~Levine, \emph{A problem in minimax variance polynomial extrapolation}, Ann.
  Math. Statist. \textbf{37} (1966), 898--903. \MR{0195215 (33 \#3418)}

\bibitem{Lloyd1952}
E.~H. Lloyd, \emph{On the estimation of variance and covariance}, Proc. Roy.
  Soc. Edinburgh. Sect. A. \textbf{63} (1952), 280--289. \MR{0048756 (14,64i)}

\bibitem{Rynne2008}
Bryan~P. Rynne and Martin~A. Youngson, \emph{Linear functional analysis},
  second ed., Springer Undergraduate Mathematics Series, Springer-Verlag London
  Ltd., London, 2008. \MR{2370216 (2008i:46001)}

\bibitem{Studden1968}
W.~J. Studden, \emph{Optimal designs on {T}chebycheff points}, Ann. Math.
  Statist \textbf{39} (1968), 1435--1447. \MR{0231497 (37 \#7050)}

\end{thebibliography}

\end{document}